\documentclass{amsart}

\usepackage{ae,aecompl}
\usepackage{esint}
\usepackage{graphicx}
\usepackage[all,cmtip]{xy}
\usepackage{accents}
\usepackage{amsmath, amscd}
\usepackage{booktabs}
\usepackage{amsthm}
\usepackage{amssymb}
\usepackage{amsfonts}
\usepackage{qsymbols}
\usepackage{latexsym}
\usepackage{mathrsfs}
\usepackage{cite}
\usepackage{color}
\usepackage{url}
\usepackage{enumerate}
\usepackage{undertilde}
\usepackage{verbatim}
\usepackage[draft=false, colorlinks=true]{hyperref}

\allowdisplaybreaks

\newcommand {\A}{{\mathcal{A}}}

\newcommand {\BMO}{{\mathrm{BMO}}}

\newcommand {\C}{{\mathbb C}}

\newcommand\cf{\alpha_{S^{n-1}}}

\newcommand {\Da}{\mathcal{J}}

\newcommand {\ud}{\mathrm{d}}
\newcommand {\ue}{\mathrm{e}}

\newcommand {\Ell}{L}

\newcommand {\F}{{\mathcal{F}}}

\newcommand {\HT}{\mathcal{H}}
\newcommand {\Hp}{\mathcal{H}^{p}_{FIO}(\Rn)}
\newcommand {\Hpp}{\mathcal{H}^{p'}_{FIO}(\Rn)}

\newcommand {\ind}{{\mathbf{1}}}

\newcommand {\rb}{\rangle}
\newcommand {\lb}{{\langle}}
\newcommand {\La}{{\mathcal{L}}}
\newcommand {\loc}{{\mathrm{loc}}}

\newcommand {\N}{{{\mathbb N}}}

\newcommand {\ph}{{\varphi}}

\newcommand {\R}{{\mathbb R}}

\newcommand {\Rn}{{\mathbb{R}^{n}}}

\newcommand {\supp}{{\mathrm{supp}}}

\newcommand {\Sp}{S^{*}(\Rn)}
\newcommand {\Spp}{S^{*}_{+}(\Rn)}

\newcommand {\Sw}{\mathcal{S}}

\newcommand {\w}{{\omega}}

\newcommand {\Z}{{{\mathbb Z}}}

\newcommand {\vanish}[1]{\relax}

\newcommand{\wh}{\widehat}
\newcommand{\wt}{\widetilde}

\makeatletter
\newcommand{\doublewidetilde}[1]{{%
  \mathpalette\double@widetilde{#1}%
}}
\newcommand{\double@widetilde}[2]{%
  \sbox\z@{$\m@th#1\widetilde{#2}$}%
  \ht\z@=.9\ht\z@
  \widetilde{\box\z@}%
}
\makeatother

\DeclareFontFamily{U}{mathx}{\hyphenchar\font45}
\DeclareFontShape{U}{mathx}{m}{n}{
      <5> <6> <7> <8> <9> <10>
      <10.95> <12> <14.4> <17.28> <20.74> <24.88>
      mathx10
      }{}
\DeclareSymbolFont{mathx}{U}{mathx}{m}{n}
\DeclareFontSubstitution{U}{mathx}{m}{n}
\DeclareMathAccent{\widecheck}{0}{mathx}{"71}

\DeclareMathOperator*{\esssup}{ess\,sup}

\newtheorem{theorem}{Theorem}[section]
\newtheorem{lemma}[theorem]{Lemma}
\newtheorem{proposition}[theorem]{Proposition}
\newtheorem{corollary}[theorem]{Corollary}

\theoremstyle{definition}
\newtheorem{definition}[theorem]{Definition}
\newtheorem{remark}[theorem]{Remark}

\numberwithin{equation}{section}
\protected\def\ignorethis#1\endignorethis{}
\let\endignorethis\relax

\title{Characterizations of Hardy spaces for Fourier integral operators}

\author{Jan Rozendaal}
\address{Mathematical Sciences Institute\\ Australian National University\\Acton ACT 2601\\Australia}
\email{janrozendaalmath@gmail.com}



\thanks{This research was supported by grant DP160100941 of the Australian Research Council.}

\begin{document}
\begin{abstract}
We prove several characterizations of the Hardy spaces for Fourier integral operators $\Hp$, for $1<p<\infty$. First we characterize $\Hp$ in terms of $L^{p}(\Rn)$-norms of parabolic frequency localizations. As a corollary, any characterization of $L^{p}(\Rn)$ yields a corresponding version for $\Hp$. In particular, we obtain a maximal function characterization and a characterization in terms of vertical square functions.
\end{abstract}

\maketitle

\section{Introduction}\label{sec:int}

It is well known that $L^{p}(\Rn)$, for $n\geq 1$ and $1<p<\infty$, can be characterized in various ways. There are characterizations in terms of maximal functions, vertical or conical square functions, and many more (see e.g.~\cite{Grafakos14a,Grafakos14b,Stein93}). For $p=1$ such expressions can be used to define the Hardy space $H^{1}(\Rn)$, and for $p=\infty$ appropriately modified versions yield $\BMO(\Rn)$. These characterizations are powerful harmonic analytic tools; for example, they can be used to show that pseudodifferential operators of order zero are bounded on $L^{p}(\Rn)$ for $1<p<\infty$.

On the other hand, it has long been known that Fourier integral operators (FIOs) of order zero are in general not bounded on $L^{p}(\Rn)$, unless $n=1$ or $p=2$. The class of Fourier integral operators extends the class of pseudodifferential operators, and FIOs arise naturally in e.g.~the analysis of wave equations and inverse problems (for more on these operators see \cite{Duistermaat11,Hormander71,Sogge17}). In fact, an FIO $T$, associated with a local canonical graph and having a compactly supported Schwartz kernel, satisfies $T:W^{s_{p},p}(\Rn)\to L^{p}(\Rn)$ for $1<p<\infty$ and $s_{p}:=(n-1)|\frac{1}{p}-\frac{1}{2}|$, and this index cannot be improved in general. This was shown by Seeger, Sogge and Stein in \cite{SeSoSt91}, extending earlier work in \cite{Peral80,Miyachi80} for the classical wave group $(e^{it\sqrt{-\Delta}})_{t\in\R}$. Nonetheless, in \cite{Smith98a} Hart Smith constructed an invariant space $\HT^{1}_{FIO}(\Rn)$ for FIOs of order zero, and this space is big enough to allow him to recover the results in \cite{SeSoSt91}. Recently, in \cite{HaPoRo18} the work of Smith was extended to a full Hardy space theory for FIOs, involving invariant spaces $\Hp$ for FIOs for all $p\in[1,\infty]$. These spaces are embedded into the $L^{p}$-scale, and among their properties we note for example that the classical wave equation is well posed on $\Hp$.

The Hardy spaces $\Hp$ for FIOs are defined in terms of a conical square function, mirroring a similar description of the classical $L^{p}$-spaces but now involving integrals over the cosphere bundle $\Sp=\Rn\times S^{n-1}$ of $\Rn$. More precisely, for $p<\infty$, the $\Hp$-norm of an $f\in\Hp$ is equivalent to the expression
\[
\|q(D)f\|_{L^{p}(\Rn)}+\Big(\int_{\Sp}\Big(\int_{0}^{1}\fint_{B_{\sqrt{\sigma}}(x,\w)}|\psi_{\nu,\sigma}(D)f(y)|^{2}\ud y\ud\nu\frac{\ud\sigma}{\sigma}\Big)^{p/2}\ud x\ud\w\Big)^{1/p}.
\]
Here $q\in C^{\infty}_{c}(\Rn)$ is such that $q(\zeta)=1$ if $|\zeta|\leq 2$, and $\psi_{\nu,\sigma}\in C^{\infty}_{c}(\Rn)$ has the property that $\psi_{\nu,\sigma}(\zeta)=0$ unless $\frac{1}{2}\sigma^{-1}\leq |\zeta|\leq 2\sigma^{-1}$ and $|\hat{\zeta}-\nu|\leq 2\sqrt{\sigma}$, where $\hat{\zeta}=\zeta/|\zeta|$ for $\zeta\neq0$. Moreover, $B_{\sqrt{\sigma}}(x,\w)\subseteq\Sp$ is a ball around $(x,\w)\in\Sp$ of radius $\sqrt{\sigma}$ with respect to an anisotropic metric that arises from contact geometry (for more details see Sections \ref{sec:tent} and \ref{subsec:wavepackets}). For $p=\infty$ appropriate modifications yield $\HT^{\infty}_{FIO}(\Rn)$.

The definition of $\Hp$ in terms of conical square functions allows one to analyze these spaces using powerful tools from harmonic analysis. In particular, the theory of tent spaces due to Coifman, Meyer and Stein \cite{CoMeSt85}, which has proved to be very effective in the analysis of elliptic and parabolic PDEs with rough coefficients (see e.g.~\cite{HoLuMiMiYa11, AuMcIRu08,Hofmann-Mayboroda09,Duong-Yan05}), can also be applied to the Hardy spaces for FIOs. This in turn yields a theory that is adapted to wave equations with rough coefficients, moving away from smooth oscillatory integral representations and instead working with suitable kernel bounds. Nonetheless, the $\Hp$ norm is a relatively involved expression that is not particularly amenable to direct calculations. Moreover, since the $\Hp$ spaces were introduced to provide an analogue of $L^{p}(\Rn)$ suited for the analysis of FIOs and wave equations, one might ask whether they can be characterized in similar ways as the classical $L^{p}$-spaces. In this article we show that this is indeed the case for $1<p<\infty$, by obtaining for every characterization of $L^{p}(\Rn)$ a corresponding one for $\Hp$. In particular, we show that $\Hp$ can be described in a simple manner using $L^{p}(\Rn)$ norms and parabolic frequency localizations. 

Our main result involves functions $\ph_{\w}\in C^{\infty}(\Rn)$, for $\w\in S^{n-1}$, that are defined in Section \ref{subsec:wavepackets}. Some of their properties are contained in Remark \ref{rem:phifunctions}, the most relevant of which for this introduction is that each $\zeta\in\supp(\ph_{\w})$ satisfies $|\zeta|\geq \frac{1}{8}$ and $|\hat{\zeta}-\w|\leq 2|\zeta|^{-1/2}$. Hence the Fourier multiplier $\ph_{\w}(D)$ localizes frequencies to a paraboloid in the direction of $\w$. Moreover, one has suitable anisotropic bounds for $\ph_{\w}$ and its derivatives, as well as for $\int_{S^{n-1}}\ph_{\nu}(\zeta)^{2}\ud\nu$ and $(\int_{S^{n-1}}\ph_{\nu}(\zeta)^{2}\ud\nu)^{-1}$ if $|\zeta|>\frac{1}{2}$. The low-frequency cutoff $q$ is as before. Our main result is then as follows.

\begin{theorem}\label{thm:intro}
Let $p\in(1,\infty)$. Then there exists a constant $C>0$ such that the following holds for all $f\in\Sw'(\Rn)$. One has $f\in\Hp$ if and only if $q(D)f\in L^{p}(\Rn)$, $\ph_{\w}(D)f\in L^{p}(\Rn)$ for almost all $\w\in S^{n-1}$, and 
\[
\Big(\int_{\Sp}|\ph_{\w}(D)f(x)|^{p}\ud x\ud\w\Big)^{\frac{1}{p}}<\infty.
\]
Moreover, 
\[
\frac{1}{C}\|f\|_{\Hp}\leq \|q(D)f\|_{L^{p}(\Rn)}+\Big(\int_{\Sp}|\ph_{\w}(D)f(x)|^{p}\ud x\ud\w\Big)^{\frac{1}{p}}\leq C\|f\|_{\Hp}
\]
if $f\in\Hp$.
\end{theorem}

Theorem \ref{thm:intro} shows that, up to a low-frequency term, elements of $\Hp$ can be described in terms of their $L^{p}$-behavior when localized to paraboloids in frequency. This description reflects the dyadic-parabolic principle that goes back to Fefferman \cite{Fefferman73b} and which has proved very effective for the analysis of FIOs (see e.g.~\cite[Chapter IX]{Stein93} and \cite{Candes-Demanet05}). The dyadic-parabolic principle is crucial in the $\Hp$ norm, and it refers to the specific scaling in the support of the functions $\psi_{\nu,\sigma}$. Because  
\begin{equation}\label{eq:introeq}
\Big(\int_{\Sp}|\ph_{\w}(D)f(x)|^{p}\ud x\ud\w\Big)^{\frac{1}{p}}=\Big(\int_{S^{n-1}}\|\ph_{\w}(D)f\|_{L^{p}(\Rn)}^{p}\ud\w\Big)^{1/p}
\end{equation}
and because $L^{p}(\Rn)$ can be decomposed in terms of frequency localizations to dyadic annuli, the dyadic-parabolic principle is again evident in Theorem \ref{thm:intro}. (For more on this see below, as well as the remarks immediately following Lemma \ref{lem:packetbounds}).

Theorem \ref{thm:intro} yields a simple description of $\Hp$, similar to the definition of $L^{p}(\Rn)$ but now involving parabolic frequency localizations in all directions. Moreover, by combining Theorem \ref{thm:intro} and \eqref{eq:introeq}, one obtains for every characterization of $L^{p}(\Rn)$ a corresponding one for $\Hp$; one simply applies the characterization to each of the functions $\ph_{\w}(D)f\in L^{p}(\Rn)$. Corollaries \ref{cor:maximal} and \ref{cor:vertical} contain two such characterizations: one in terms of maximal functions and one in terms of vertical square functions. 

To prove Theorem \ref{thm:intro}, which is contained in the main text as Theorem \ref{thm:chargeneral}, we proceed in two steps. First we show, in Section \ref{sec:wavepackets}, that the functions $\psi_{\nu,\sigma}$ in the $\Hp$ norm can be replaced by a different expression involving the functions $\ph_{\nu}$. More precisely, the function $\psi_{\nu,\sigma}$, introduced in Section \ref{subsec:wavepackets}, is the product of a term that localizes to a dyadic annulus of radius $\sigma^{-1}$ and a term that localizes to a cone of aperture $\sqrt{\sigma}$ in the direction of $\nu$. To prove Theorem \ref{thm:intro} it is more convenient to replace the term which localizes to a cone by the function $\ph_{\nu}$, which localizes to a parabola in the direction of $\nu$. This allows one to treat the parameters $\nu$ and $\sigma$ separately, thereby decoupling the dyadic and parabolic localizations. 

In the next step, in Section \ref{sec:char}, we apply the conical square function characterization of $L^{p}(\Rn)$ to \eqref{eq:introeq}. This yields an expression involving the appropriate localization in frequency, but the relevant integral averages are taken over isotropic balls in $\Rn$, instead of the anisotropic balls in $\Sp$ that occur in the $\Hp$ norm. It then remains to move from isotropic averages on $\Rn$ to anisotropic averages on $\Sp$, which requires some delicate technical estimates. It is an open question whether a suitably modified version of Theorem \ref{thm:intro} characterizes $\HT^{1}_{FIO}(\Rn)$ and $\HT^{\infty}_{FIO}(\Rn)$ (see Remark \ref{rem:proof}). 

This article is organized as follows. In Section \ref{sec:tent} we collect some basics on tent spaces over the cosphere bundle, as is necessary to define $\Hp$. In Section \ref{sec:wavepackets} we then introduce the wave packets that are used to localize frequencies in the $\Hp$-norm. We first define $\Hp$ using the wave packets from \cite{HaPoRo18}, and then we show that one can also work with wave packets involving the functions $\ph_{\nu}$. In Section \ref{subsec:changes} we prove a technical lemma which is used in the proof of Theorem \ref{thm:intro} to change between different wave packets. In Section \ref{sec:char} we then prove Theorem \ref{thm:intro}, and we derive from it a maximal function characterization and a vertical square function characterization of $\Hp$.

\subsection{Notation}\label{sec:not}

The natural numbers are $\N=\{1,2,\ldots\}$, and $\Z_{+}=\N\cup\{0\}$. Throughout this article we fix $n\in\N$ with $n\geq2$. With some minor modifications, the techniques in this article also apply for $n=1$. However, $\HT^{p}_{FIO}(\R)=L^{p}(\R)$ for $p\in(1,\infty)$ (see \cite[Theorem 7.4]{HaPoRo18}), and in this case our results are classical.

For $\xi,\eta\in\Rn$ we write $\lb\xi\rb:=(1+|\xi|^{2})^{1/2}$ and $\lb \xi,\eta\rb:=\xi\cdot\eta$, and if $\xi\neq 0$ then $\hat{\xi}:=\xi/|\xi|$. We use multi-index notation, so that $\partial^{\alpha}_{\xi}=\partial^{\alpha_{1}}_{\xi_{1}}\ldots\partial^{\alpha_{n}}_{\xi_{n}}$ for $\xi=(\xi_{1},\ldots,\xi_{n})\in\Rn$ and $\alpha=(\alpha_{1},\ldots,\alpha_{n})\in\Z_{+}^{n}$.

The space of Schwartz functions on $\Rn$ is $\Sw(\Rn)$, and the tempered distributions are $\Sw'(\Rn)$. The Fourier transform of $f\in\Sw'(\Rn)$ is denoted by $\F f$ or $\widehat{f}$, and the inverse Fourier transform by $\F^{-1}f$. If $f\in\Ell^{1}(\Rn)$ then
\begin{align*}
\F f(\xi)=\int_{\Rn}\ue^{-i x\cdot \xi}f(x)\ud x\quad (\xi\in\Rn).
\end{align*}
For $m:\Rn\to\C$ a measurable function of temperate growth, $m(D)$ is the Fourier multiplier with symbol $m$. 

The volume of a measurable subset $B$ of a measure space $(\Omega,\mu)$ is $|B|$. For an integrable $F:B\to\C$, we write
\[
\fint_{B}F(x)\ud \mu(x)=\frac{1}{|B|}\int_{B}F(x)\ud \mu(x)
\]
if $|B|<\infty$. The H\"{o}lder conjugate of $p\in[1,\infty]$ is denoted by $p'$, and the indicator function of a set $E$ is $\ind_{E}$. 

The space of continuous linear operators on a Banach space $X$ is $\La(X)$. We write $f(s)\lesssim g(s)$ to indicate that $f(s)\leq Cg(s)$ for all $s$ and a constant $C\geq0$ independent of $s$, and similarly for $f(s)\gtrsim g(s)$ and $g(s)\eqsim f(s)$.

\section{Tent spaces over the cosphere bundle}\label{sec:tent}

In this section we collect some preliminaries on the tent spaces that are used to define the Hardy spaces for Fourier integral operators. The underlying metric measure space is the cosphere bundle $\Sp:=\Rn\times S^{n-1}$ over $\Rn$, endowed with the standard measure $\ud x\ud\w$ and with a metric that arises from contact geometry. We give the relevant definitions here, but for readers unfamiliar with contact geometry we note that we use only two properties of the metric:
\begin{itemize}
\item that it has a convenient equivalent expression, in \eqref{eq:metric};
\item that $(\Sp,d,\ud x\ud\w)$ is a doubling metric measure space, cf.~\eqref{eq:doubling}.
\end{itemize}

Throughout, we denote elements of $S^{n-1}$ by $\w$ or $\nu$. Let $g_{S^{n-1}}$ be the standard Riemannian metric on $S^{n-1}$. The cosphere bundle $\Sp$ is a contact manifold with respect to the standard contact form $\alpha_{S^{n-1}}:=\w\cdot dx$, the kernel of which is a smooth distribution of codimension $1$ subspaces of the tangent bundle $T(\Sp)$ of $\Sp$. For $(x,\w),(y,\nu)\in\Sp$, set
\[
d((x, \omega), (y, \nu)) := \inf_{\gamma} \int_0^1 |\gamma'(s)|\ud s,
\]
where the infimum is taken over all piecewise $C^1$ curves $\gamma : [0,1] \to \Sp$ such that $\gamma(0) = (x, \omega)$, $\gamma(1) = (y, \nu)$, and $\cf(\gamma'(s)) = 0$ for almost all $s \in [0,1]$. Here $|\gamma'(s)|$ is the length of the tangent vector $\gamma'(s)$ with respect to the product metric $dx^2 + g_{S^{n-1}}$.

We will mostly work with an equivalent analytic expression for $d$. By \cite[Lemma 2.1]{HaPoRo18}, one has
\begin{equation}\label{eq:metric}
d((x,\w),(y,\nu))\eqsim \big(|x-y|^{2}+|\lb\w,x-y\rb|+|\w-\nu|^{2}\big)^{1/2}
\end{equation}
for all $(x,\w),(y,\nu)\in\Sp$, where the implicit constants only depend on $n$. Moreover, we will frequently use that, by \cite[Lemma 2.3]{HaPoRo18}, there exists a constant $C=C(n)>0$ with the following property. For all $\lambda\geq1$, $\tau>0$ and $(x,\w)\in\Sp$, one has
\begin{equation}\label{eq:doubling}
\frac{1}{C}\lambda^{n}|B_{\tau}(x,\w)|\leq |B_{\lambda\tau}(x,\w)|\leq C\lambda^{2n}|B_{\tau}(x,\w)|.
\end{equation}
Here and throughout, the volume $|U|$ of $U\subseteq\Sp$ is taken with respect to the standard measure $\ud x\ud\w$. In particular, the second inequality in \eqref{eq:doubling} shows that $(\Sp,d,\ud x\ud\w)$ is a doubling metric measure space. We note that the different powers of $\lambda$ in \eqref{eq:doubling} correspond to the distinct behavior of small and large balls in $\Sp$. For $\tau>0$ small one has $|B_{\tau}(x,\w)|\eqsim \tau^{2n}$, and for $\tau$ large $|B_{\tau}(x,\w)|\eqsim \tau^{n}$.

We can now define tent spaces over $\Sp$. Let $\Spp:=\Sp\times(0,\infty)$, endowed with the measure $\ud x\ud\w\frac{\ud\sigma}{\sigma}$. For $F\in L^{2}_{\loc}(\Spp)$ and $(x,\w)\in\Sp$, set
\begin{equation}\label{eq:pnorm}
\A F(x,\w):=\Big(\int_{0}^{\infty}\fint_{B_{\sqrt{\sigma}}(x,\w)}|F(y,\nu,\sigma)|^{2}\ud y\ud \nu\frac{\ud \sigma}{\sigma}\Big)^{1/2}\in[0,\infty]
\end{equation}
and
\begin{equation}\label{eq:inftynorm}
\mathcal{C}F(x,\w):=\sup_{B}\Big(\frac{1}{|B|}\int_{T(B)}|F(y,\nu,\sigma)|^{2}\ud y\ud \nu\frac{\ud \sigma}{\sigma}\Big)^{1/2}\in[0,\infty],
\end{equation}
where the supremum is taken over all balls $B\subseteq \Sp$ containing $(x,\w)$, and 
\begin{equation}\label{eq:tent}
T(B):=\{(y,\nu,\sigma)\in\Spp\mid d((y,\nu),B^{c})\geq \sqrt{\sigma}\}
\end{equation}
is the tent over $B$. 

\begin{definition}\label{def:tent spaces}
For $p\in[1,\infty)$, the \emph{tent space} $T^{p}(\Sp)$ consists of all $F\in L^{2}_{\loc}(\Spp)$ such that $\A F\in L^{p}(\Sp)$, endowed with the norm
\[
\|F\|_{T^{p}(\Sp)}:=\|\A F\|_{L^{p}(\Sp)}.
\]
Also, $T^{\infty}(\Sp)$ consists of all $F\in L^{2}_{\loc}(\Spp)$ such that $\mathcal{C}F\in L^{\infty}(\Sp)$, with 
\[
\|F\|_{T^{\infty}(\Sp)}:=\|\mathcal{C}F\|_{L^{\infty}(\Sp)}.
\]
\end{definition}

For all $p\in[1,\infty]$, the tent space $T^{p}(\Sp)$ is a Banach space. For more on the theory of tent spaces see e.g.~\cite{CoMeSt85,Amenta14}. For completeness we note that we in fact consider parabolic tent spaces, due to the factor $\sqrt{\sigma}$ in \eqref{eq:pnorm} and \eqref{eq:tent}. This makes no difference for the theory itself, as can be shown using a simple rescaling (see \cite{AuKrMoPo12}). 

For our main theorem we will need the following quantitative change of aperture formula.

\begin{lemma}\label{lem:aperture}
Let $p\in[1,\infty)$. Then there exists a $C=C(n,p)\geq0$ such that, for all $\lambda\in[1,\infty)$ and $F\in T^{p}(\Sp)$, one has
\[
\Big(\int_{\Sp}\Big(\int_{0}^{\infty}\fint_{B_{\lambda\sqrt{\sigma}}(x,\w)}|F(y,\nu,\sigma)|^{2}\ud y\ud\nu\frac{\ud\sigma}{\sigma}\Big)^{\frac{p}{2}}\ud x\ud\w\Big)^{\frac{1}{p}}\leq C\lambda^{2n\tau_{p}}\|F\|_{T^{p}(\Sp)}.
\]
Here $\tau_{p}=0$ for $p\geq 2$, and $\tau_{p}=\frac{1}{p}-\frac{1}{2}$ for $p<2$.
\end{lemma}
\begin{proof}
With some minor modifications, one can use the argument in \cite{Auscher11}. There the statement is proved, together with a reverse inequality, for tent spaces over $\Rn$ that have a different scaling in their norm. The proof of the statement which we need only relies on the second inequality in \eqref{eq:doubling} and on some basic theory of tent spaces over doubling metric measure spaces.
\end{proof}

\section{Wave packets and Hardy spaces for FIOs}\label{sec:wavepackets}

In this section we introduce the Hardy spaces for FIOs in terms of wave packets. Wave packets are functions that are suitably localized in both the position and momentum variables. We first define the Hardy spaces for FIOs using the wave packets from \cite{HaPoRo18}. These wave packets have a dyadic-parabolic localization built into them. Then we introduce new wave packets for which the parabolic and dyadic localizations are decoupled, and we show that the Hardy spaces for FIOs can also be described in terms of these wave packets. In the next section we will use the new wave packets, and specifically their decoupling of parabolic and dyadic localizations, to prove our main theorem. Also, in Section \ref{subsec:changes} we prove a technical lemma that will allow us to move between different choices of wave packets.

\subsection{Wave packets}\label{subsec:wavepackets}

We first recall the definition of the wave packets from \cite{HaPoRo18}. Fix a non-negative radial $\ph\in C^{\infty}_{c}(\Rn)$ such that $\phi\equiv1$ in a neighborhood of zero, and $\ph(\zeta)=0$ for $|\zeta|>1$. For $\w\in S^{n-1}$, $\sigma>0$ and $\zeta\in\Rn\setminus\{0\}$, set $\ph_{\w,\sigma}(\zeta):=c_{\sigma}\ph\big(\tfrac{\hat{\zeta}-\w}{\sqrt{\sigma}}\big)$, where $c_{\sigma}:=\big(\int_{S^{n-1}}\ph\big(\tfrac{e_{1}-\nu}{\sqrt{\sigma}}\big)^{2}\ud\nu\big)^{-1/2}$ for $e_{1}$ the first basis vector of $\Rn$ (this choice is immaterial). Also set $\ph_{\w,\sigma}(0):=0$. Next, fix a non-negative radial $\Psi\in C^{\infty}_{c}(\Rn)$ such that $\Psi(\zeta)=0$ if $|\zeta|\notin[\frac{1}{2},2]$, and
\begin{equation}\label{eq:Psi}
\int_{0}^{\infty}\Psi(\sigma\zeta)^{2}\frac{\ud \sigma}{\sigma}=1\quad(\zeta\neq0).
\end{equation}
Set $\Psi_{\sigma}(\zeta):=\Psi(\sigma\zeta)$ and $\psi_{\w,\sigma}(\zeta):=\Psi_{\sigma}(\zeta)\ph_{\w,\sigma}(\zeta)$ for $\w\in S^{n-1}$, $\sigma>0$ and $\zeta\in\Rn$. Let
\[
r(\zeta):=\Big(\int_{1}^{\infty}\Psi_{\sigma}(\zeta)^{2}\frac{\ud\sigma}{\sigma}\Big)^{1/2}\quad(\zeta\neq0)
\]
and $r(0):=1$. It is straightforward to prove that $r\in C^{\infty}_{c}(\Rn)$, using that $r$ is radial, by showing that all the derivatives of $r$ vanish where $r(\zeta)=0$. 

Next, we associate a wave packet transform with these wave packets. Such transforms have long been effective tools in microlocal analysis (see e.g.~\cite{Cordoba-Fefferman78,Folland89,Martinez02}). For $f\in \Sw'(\Rn)$ and $(x,\w,\sigma)\in\Spp$, our wave packet transform is given by
\[
Wf(x,\w,\sigma):=\begin{cases}\psi_{\w,\sigma}(D)f(x)&\text{if }\sigma\in(0,1),\\
|S^{n-1}|^{-1/2}\ind_{[1,e]}(\sigma)r(D)f(x)&\text{if }\sigma\geq1.\end{cases}
\]
For later use we also write 
\[
W_{\sigma}f(x,\w):=Wf(x,w,\sigma)
\]
for the transform that maps functions on $\Rn$ to functions on $\Sp$ at a fixed scale $\sigma$. We can now define the Hardy spaces for Fourier integral operators in terms of the tent spaces $T^{p}(\Sp)$, $p\in[1,\infty]$, from Definition \ref{def:tent spaces}.

\begin{definition}\label{def:Hardy}
Let $p\in[1,\infty]$. Then $\Hp$ consists of all $f\in\Sw'(\Rn)$ such that $Wf\in T^{p}(\Sp)$, endowed with the norm
\[
\|f\|_{\Hp}:=\|Wf\|_{T^{p}(\Sp)}
\]
for $f\in\Hp$.
\end{definition}

For all $p\in[1,\infty]$, up to norm equivalence, $\Hp$ is independent of the choice of $\ph$ and $\Psi$, by \cite[Proposition 6.4]{HaPoRo18}. We note that, by \cite[Corollary 7.6]{HaPoRo18}, an equivalent norm on $\Hp$ is given by the expression
\[
\|q(D)f\|_{L^{p}(\Rn)}+\Big(\int_{\Sp}\Big(\int_{0}^{1}\fint_{B_{\sqrt{\sigma}}(x,\w)}|\psi_{\nu,\sigma}(D)f(y)|^{2}\ud y\ud\nu\frac{\ud\sigma}{\sigma}\Big)^{p/2}\ud x\ud\w\Big)^{1/p}
\]
from the introduction, where $q\in C^{\infty}_{c}(\Rn)$ is such that $q(\zeta)=1$ for $|\zeta|\leq 2$.

We now introduce new wave packets. For $\w\in S^{n-1}$ and $\zeta\in\Rn$, set 
\begin{equation}\label{eq:phiomega}
\ph_{\w}(\zeta):=\int_{0}^{4}\psi_{\w,\tau}(\zeta)\frac{\ud\tau}{\tau}.
\end{equation}
In addition, for $\sigma\in(0,1)$, let
\[
\theta_{\w,\sigma}(\zeta):=\Psi_{\sigma}(\zeta)\ph_{\w}(\zeta)
\]
and
\[
\chi_{\w,\sigma}(\zeta):=\begin{cases}
(\int_{S^{n-1}}\ph_{\nu}(\zeta)^{2}\ud\nu)^{-1}\theta_{\w,\sigma}(\zeta)&\text{for }\zeta\in\supp(\theta_{\w,\sigma}),\\
0&\text{otherwise}.
\end{cases}
\]
We collect some properties of these wave packets in the following lemma. 

\begin{lemma}\label{lem:packetbounds}
For $\w\in S^{n-1}$ and $\sigma\in(0,1)$, let $\eta_{\w,\sigma}\in\{\psi_{\w,\sigma},\theta_{\w,\sigma},\chi_{\w,\sigma}\}$. Then $\eta_{\w,\sigma}\in C^{\infty}_{c}(\Rn)$, and each $\zeta\in\supp(\eta_{\w,\sigma})$ satisfies 
\begin{equation}\label{eq:dyadicparabolic}
\frac{1}{2}\sigma^{-1}\leq |\zeta|\leq 2\sigma^{-1}\text{ and }|\hat{\zeta}-\w|\leq 2\sqrt{\sigma}.
\end{equation}
Moreover, for all $\alpha\in\Z_{+}^{n}$ and $\beta\in\Z_{+}$ there exists a constant $C=C(\alpha,\beta)\geq0$ such that
\begin{equation}\label{eq:boundspsi}
|\lb\w,\nabla_{\zeta}\rb^{\beta}\partial_{\zeta}^{\alpha}\eta_{\w,\sigma}(\zeta)|\leq C\sigma^{-\frac{n-1}{4}+\frac{|\alpha|}{2}+\beta}
\end{equation}
for all $(\zeta,\w,\sigma)\in\Spp$. Also, for each $N\geq0$ there exists a $C_{N}\geq0$ such that
\begin{equation}\label{eq:boundspsiinverse}
|\F^{-1}(\eta_{\w,\sigma})(x)|\leq C_{N}\sigma^{-\frac{3n+1}{4}}(1+\sigma^{-1}|x|^{2}+\sigma^{-2}\lb\w,x\rb^{2})^{-N}
\end{equation}
for all $(x,\w,\sigma)\in\Spp$.
\end{lemma}

We note that the wave packets $\psi_{\w,\sigma}$, $\theta_{\w,\sigma}$ and $\chi_{\w,\sigma}$ are all localized to the dyadic-parabolic region in the direction of $\w$ and at scale $\sigma$ that is defined by \eqref{eq:dyadicparabolic}. Moreover, their inverse Fourier transforms decay rapidly off an inverted dyadic-parabolic region. However, for $\psi_{\w,\sigma}$ this localization is linked to both the $\w$ and $\sigma$ variables simultaneously, through the function $\ph_{\w,\sigma}$, whereas for $\theta_{\w,\sigma}$ and $\chi_{\w,\sigma}$ it is decoupled into the dyadic localization enforced by $\Psi_{\sigma}$ and the parabolic localization coming from $\ph_{\w}$.
\begin{proof}
For $\eta_{\w,\sigma}=\psi_{\w,\sigma}$, the required statements are contained in \cite[Lemma 4.1]{HaPoRo18} for all $\w\in S^{n-1}$ and $\sigma>0$. It is also shown there that 
\begin{equation}\label{eq:csigma}
c_{\sigma}=\Big(\int_{S^{n-1}}\ph\big(\tfrac{e_{1}-\nu}{\sqrt{\sigma}}\big)^{2}\ud\nu\Big)^{-1/2}\eqsim \sigma^{-\frac{n-1}{4}}.
\end{equation}

Next, fix $\w\in S^{n-1}$ and $\sigma\in(0,1)$. For $\eta_{\w,\sigma}=\theta_{\w,\sigma}$ the first statement follows from the support properties of $\ph$ and $\Psi$, by noting that
\[
\theta_{\w,\sigma}=\int^{4\sigma}_{\sigma/4}\Psi_{\sigma}\Psi_{\tau}\ph_{\w,\tau}\frac{\ud\tau}{\tau}.
\]
Moreover, \eqref{eq:boundspsi} is a consequence of the corresponding statement for $\psi_{\w,\tau}$, $\tau\in(\sigma/4,4\sigma)$. Now \eqref{eq:boundspsiinverse} is obtained by integrating by parts with respect to the operator
\[
L:=(1+\sigma^{-1}|x|^{2}+\sigma^{-2}\lb\w,x\rb^{2})^{-1}(1-\sigma^{-1}\Delta_{\zeta}-\sigma^{-2}\lb\w,\nabla_{\zeta}\rb^{2})
\]
in the expression
\[
\F^{-1}(\theta_{\w,\sigma})(x)=(2\pi)^{-n}\int_{\Rn}e^{i x\cdot\zeta}\theta_{\w,\sigma}(\zeta)\ud\zeta\quad(x\in\Rn),
\]
using the bounds in \eqref{eq:boundspsi} and the support properties of $\theta_{\w,\sigma}$. For more on this see the proof of \cite[Lemma 4.1]{HaPoRo18}.

Finally, for $\eta_{\w,\sigma}=\chi_{\w,\sigma}$ we need appropriate bounds for the derivatives of $(\int_{S^{n-1}}\ph_{\nu}^{2}\ud\nu)^{-1}$ on $\supp(\theta_{\w,\sigma})$. To this end, let $c>0$ be such that $\phi(\zeta)=1$ for $|\zeta|\leq c$, and let $\zeta\in\Rn$ be such that $|\zeta|\geq \frac{1}{2}$. Set $E_{\zeta}:=\{\nu\in S^{n-1}\mid |\nu-\hat{\zeta}|\leq c|2\zeta|^{-1/2}\}$ and $F_{\zeta}:=\{\nu\in S^{n-1}\mid |\nu-\hat{\zeta}|\leq \sqrt{2}|\zeta|^{-1/2}\}$. Then 
\[
|E_{\zeta}|\eqsim |F_{\zeta}|\eqsim |\zeta|^{-\frac{n-1}{2}}.
\]
Note also that 
\[
\psi_{\nu,\tau}(\zeta)=c_{\tau}\Psi_{\tau}(\zeta)\ph\big(\tfrac{\hat{\zeta}-\nu}{\sqrt{\tau}}\big)=c_{\tau}\Psi_{\tau}(\zeta)
\]
for $\nu\in E_{\zeta}$ and $\tau\geq |2\zeta|^{-1}$, and that 
\[
\psi_{\nu,\tau}(\zeta)=c_{\tau}\Psi_{\tau}(\zeta)\ph\big(\tfrac{\hat{\zeta}-\nu}{\sqrt{\tau}}\big)=0
\]
for all $\nu\notin F_{\zeta}$ and $\tau>0$, where we used that $\Psi_{\tau}(\zeta)=0$ if $|\zeta|\notin [\frac{1}{2}\tau^{-1},2\tau^{-1}]$. Since $|\zeta|\geq \frac{1}{2}$, one has
\[
\ph_{\nu}(\zeta)=\int_{|2\zeta|^{-1}}^{\min(4,2|\zeta|^{-1})}\psi_{\nu,\tau}(\zeta)\frac{\ud\tau}{\tau}=\int_{|2\zeta|^{-1}}^{2|\zeta|^{-1}}\psi_{\nu,\tau}(\zeta)\frac{\ud\tau}{\tau}.
\]
We now combine all of this with \eqref{eq:Psi} and \eqref{eq:csigma} to obtain
\begin{align*}
&1=\Big(\int_{0}^{\infty}\Psi_{\tau}(\zeta)^{2}\frac{\ud\tau}{\tau}\Big)^{2}=\Big(\int_{|2\zeta|^{-1}}^{2|\zeta|^{-1}}\Psi_{\tau}(\zeta)^{2}\frac{\ud\tau}{\tau}\Big)^{2}\lesssim \Big(\int_{|2\zeta|^{-1}}^{2|\zeta|^{-1}}\Psi_{\tau}(\zeta)\frac{\ud\tau}{\tau}\Big)^{2}\\
&\eqsim|\zeta|^{\frac{n-1}{2}}\int_{E_{\zeta}}\Big(\int_{|2\zeta|^{-1}}^{2|\zeta|^{-1}}\Psi_{\tau}(\zeta)\frac{\ud\tau}{\tau}\Big)^{2}\ud\nu\eqsim \int_{E_{\zeta}}\Big(\int_{|2\zeta|^{-1}}^{2|\zeta|^{-1}}c_{\tau}\Psi_{\tau}(\zeta)\frac{\ud\tau}{\tau}\Big)^{2}\ud\nu\\\
&\leq \int_{S^{n-1}}\Big(\int_{|2\zeta|^{-1}}^{2|\zeta|^{-1}}\psi_{\nu,\tau}(\zeta)\frac{\ud\tau}{\tau}\Big)^{2}\ud\nu=\int_{S^{n-1}}\ph_{\nu}(\zeta)^{2}\ud\nu\\
&=\int_{F_{\zeta}}\Big(\int_{|2\zeta|^{-1}}^{2|\zeta|^{-1}}\psi_{\nu,\tau}(\zeta)\frac{\ud\tau}{\tau}\Big)^{2}\ud\nu\lesssim \int_{F_{\zeta}}\Big(\int_{|2\zeta|^{-1}}^{2|\zeta|^{-1}}\tau^{-\frac{n-1}{4}}\frac{\ud\tau}{\tau}\Big)^{2}\ud\nu\\
&\lesssim |\zeta|^{\frac{n-1}{2}}\int_{F_{\zeta}}\Big(\int_{|2\zeta|^{-1}}^{2|\zeta|^{-1}}\frac{\ud\tau}{\tau}\Big)^{2}\ud\nu\lesssim 1.
\end{align*}
In particular,
\begin{equation}\label{eq:boundbothways}
\int_{S^{n-1}}\ph_{\nu}(\zeta)^{2}\ud\nu\eqsim 1
\end{equation}
for implicit constants independent of $\zeta\in\Rn$ with $|\zeta|\geq \frac{1}{2}$. By combining similar arguments with the bounds for the derivatives of $\psi_{\nu,\tau}$, one can show that
\[
\Big|\lb\hat{\zeta},\nabla_{\zeta}\rb^{\beta}\partial_{\zeta}^{\alpha}\Big(\int_{S^{n-1}}\ph_{\nu}(\zeta)^{2}\ud\nu\Big)^{-1}\Big|\lesssim |\zeta|^{-\frac{|\alpha|}{2}-\beta}
\]
for all $\alpha\in\Z_{+}^{n}$, $\beta\in\Z_{+}$ and $\zeta\in\Rn$ with $|\zeta|\geq \frac{1}{2}$. In fact, because
\[
\int_{S^{n-1}}\ph_{\nu}(\zeta)^{2}\ud\nu=\int_{S^{n-1}}\Big(\int_{0}^{4}c_{\tau}^{2}\Psi(\tau\zeta)^{2}\ph\big(\tfrac{\hat{\zeta}-\nu}{\sqrt{\tau}}\big)\frac{\ud\tau}{\tau}\Big)^{2}\ud\nu
\]
for all $\zeta\in\Rn$ and because $\ph$ and $\Psi$ are radial, $\int_{S^{n-1}}\ph_{\nu}^{2}\ud\nu$ is radial as well. Hence one in fact has
\begin{equation}\label{eq:boundbelow}
\Big|\partial_{\zeta}^{\alpha}\Big(\int_{S^{n-1}}\ph_{\nu}(\zeta)^{2}\ud\nu\Big)^{-1}\Big|\lesssim |\zeta|^{-|\alpha|}
\end{equation}
for all $\alpha\in\Z_{+}^{n}$ and $\zeta\in\Rn$ with $|\zeta|\geq \frac{1}{2}$. 

Now, since for all $\zeta\in\supp(\theta_{\w,\sigma})$ one has $2\sigma^{-1}\geq |\zeta|\geq \frac{1}{2}\sigma^{-1}>\frac{1}{2}$, it follows from \eqref{eq:boundbelow} that
\begin{equation}\label{eq:boundbelow2}
\Big|\partial_{\zeta}^{\alpha}\Big(\int_{S^{n-1}}\ph_{\nu}(\zeta)^{2}\ud\nu\Big)^{-1}\Big|\lesssim \sigma^{|\alpha|}
\end{equation}
for all $\alpha\in\Z_{+}^{n}$ and $\zeta\in\supp(\theta_{\w,\sigma})$. This in turn implies that $\chi_{\w,\sigma}\in C^{\infty}_{c}(\Rn)$ with $\supp(\chi_{\w,\sigma})=\supp(\theta_{\w,\sigma})$. And by combining \eqref{eq:boundbelow2} with what we have already shown, we obtain \eqref{eq:boundspsi} for $\eta_{\w,\sigma}=\chi_{\w,\sigma}$. Then \eqref{eq:boundspsiinverse} follows as before.
\end{proof}

\begin{remark}\label{rem:phifunctions}
The collection $(\ph_{\w})_{\w\in S^{n-1}}\subseteq C^{\infty}(\Rn)$ has the following properties:
\begin{enumerate}
\item\label{it:function1} For all $\w\in S^{n-1}$ and $\zeta\neq0$ one has $\ph_{\w}(\zeta)=0$ if $|\zeta|<\frac{1}{8}$ or $|\hat{\zeta}-\w|>2|\zeta|^{-1/2}$;
\item\label{it:function2} For all $\alpha\in\Z_{+}^{n}$ and $\beta\in\Z_{+}$ there exists a $C=C(\alpha,\beta)\geq0$ such that
\[
|\lb\w,\nabla_{\zeta}\rb^{\beta}\partial^{\alpha}_{\zeta}\ph_{\w}(\zeta)|\leq C|\zeta|^{\frac{n-1}{4}-\frac{|\alpha|}{2}-\beta}
\]
for all $\w\in S^{n-1}$ and $\zeta\neq0$;
\item\label{it:function3} For all $\alpha\in\Z_{+}^{n}$ there exists a $C=C(\alpha,\beta)\geq0$ such that
\begin{equation}\label{eq:phiomega1}
\Big|\partial_{\zeta}^{\alpha}\Big(\int_{S^{n-1}}\ph_{\nu}(\zeta)^{2}\ud\nu\Big)\Big|\leq C |\zeta|^{-|\alpha|}
\end{equation}
for all $\zeta\in\Rn$, and 
\begin{equation}\label{eq:phiomega2}
\Big|\partial_{\zeta}^{\alpha}\Big(\int_{S^{n-1}}\ph_{\nu}(\zeta)^{2}\ud\nu\Big)^{-1}\Big|\leq C |\zeta|^{-|\alpha|}
\end{equation}
if $|\zeta|\geq\frac{1}{2}$. 
\end{enumerate}
Indeed, \eqref{it:function1} and \eqref{it:function2} follow from the properties of $\psi_{\w,\tau}$, for $\tau>0$, by writing 
\[
\ph_{\w}(\zeta)=\int_{|2\zeta|^{-1}}^{\min(4,2|\zeta|^{-1})}\psi_{\w,\tau}(\zeta)\frac{\ud\tau}{\tau}=\int_{|2\zeta|^{-1}}^{\min(4,2|\zeta|^{-1})}\Psi_{\tau}(\zeta)c_{\tau}\ph\big(\tfrac{\hat{\zeta}-\w}{\tau}\big)\frac{\ud\tau}{\tau}
\]
for $|\zeta|\geq\frac{1}{8}$. Moreover, \eqref{eq:phiomega2} is just \eqref{eq:boundbelow}, and \eqref{eq:phiomega1} is derived from \eqref{eq:boundbothways} and the bounds on $\psi_{\nu,\tau}$ in the same manner as \eqref{eq:boundbelow}.
\end{remark}

\begin{remark}\label{rem:support}
The specific support assumptions that we have made on $\ph$ and $\Psi$ can be weakened. Doing so would lead to different support properties of the wave packets in Lemma \ref{lem:packetbounds}, but with minor modifications all the results in this article extend to more general choices of $\ph$ and $\Psi$.
\end{remark}

\subsection{Wave packet transforms}

Here we introduce wave packet transforms associated with the new wave packets $\{\theta_{\w,\sigma}\}_{\w,\sigma}$ and $\{\chi_{\w,\sigma}\}_{\w,\sigma}$. Then we show that these transforms yield equivalent norms on the Hardy spaces for FIOs.

Set 
\begin{equation}\label{eq:s}
s(\zeta):=1-\int_{0}^{1}\int_{S^{n-1}}\theta_{\w,\sigma}(\zeta)\chi_{\w,\sigma}(\zeta)\ud\w\frac{\ud\sigma}{\sigma}\quad(\zeta\in\Rn).
\end{equation}
Note that $s\in C^{\infty}_{c}(\Rn)$ with $s(\zeta)=0$ for $|\zeta|>2$, since for such $\zeta$ one can use \eqref{eq:Psi}:
\begin{align*}
s(\zeta)&=1-\Big(\int_{S^{n-1}}\ph_{\nu}(\zeta)^{2}\ud\nu\Big)^{-1}\int_{0}^{1}\int_{S^{n-1}}\Psi_{\sigma}(\zeta)^{2}\ph_{\w}(\zeta)^{2}\ud\w\frac{\ud\sigma}{\sigma}\\
&=1-\int_{|2\zeta|^{-1}}^{2|\zeta|^{-1}}\Psi_{\sigma}(\zeta)^{2}\frac{\ud\sigma}{\sigma}=0.
\end{align*}
Let $h\in C^{\infty}_{c}(\Rn)$ be such that $h\equiv 1$ on $\supp(s)$. Next, for $f\in \Sw'(\Rn)$ and $(x,\w,\sigma)\in\Spp$, set
\[
Vf(x,\w,\sigma):=\begin{cases}\theta_{\w,\sigma}(D)f(x)&\text{if }\sigma\in(0,1),\\
|S^{n-1}|^{-1/2}\ind_{[1,e]}(\sigma)s(D)f(x)&\text{if }\sigma\geq1,\end{cases}
\]
and
\[
Uf(x,\w,\sigma):=\begin{cases}\chi_{\w,\sigma}(D)f(x)&\text{if }\sigma\in(0,1),\\
|S^{n-1}|^{-1/2}\ind_{[1,e]}(\sigma)h(D)f(x)&\text{if }\sigma\geq1.\end{cases}
\]
We also write
\[
V_{\sigma}f(x,\w):=Vf(x,\w,\sigma)\quad\text{and}\quad U_{\sigma}f(x,\w):=Uf(x,\w,\sigma).
\] 
These transforms have many of the same properties as $W$ and $W_{\sigma}$, contained in \cite[Propositions 4.2 and 4.3]{HaPoRo18}. Some of these properties are formulated in terms of a class $\Da(\Spp)$ of test functions on $\Spp$, and the corresponding class $\Da'(\Spp)$ of distributions. More precisely, $\Da(\Spp)$ is the collection of $F\in L^{\infty}(\Spp)$ such that 
\[
(x,\w,\sigma)\mapsto (1+|x|+\sigma+\sigma^{-1})^{N}F(x,\w,\sigma)
\]
is an element of $L^{\infty}(\Spp)$ for all $N\in\Z_{+}$, endowed with the topology generated by the corresponding weighted $L^{\infty}$-norms. Then $\Da'(\Spp)$ is the space of continuous linear $G:\Da(\Spp)\to \C$, endowed with the topology induced by $\Da(\Spp)$. If $G\in L^{1}_{\loc}(\Spp)$ is such that 
\begin{equation}\label{eq:duality}
F\mapsto \int_{\Spp}F(x,\w,\sigma)\overline{G(x,\w,\sigma)}\ud x\ud\w\frac{\ud\sigma}{\sigma}
\end{equation}
defines an element of $\Da'(\Spp)$, then we identify $G$ with the corresponding element of $\Da'(\Spp)$. We write $\lb G,F\rb_{\Spp}$ for the duality between $G\in \Da'(\Spp)$ and $F\in \Da(\Spp)$. These classes play only a minor role in what follows. However, they allow us to extend some considerations from functions on $\Rn$ to elements of $\Sw'(\Rn)$, using the following lemma.

\begin{lemma}\label{lem:transforms}
The following maps are continuous:
\begin{itemize}
\item $W,V,U:L^{2}(\Rn)\to L^{2}(\Spp,\ud x\ud\w\frac{\ud\sigma}{\sigma})$;
\item $W,V,U:\Sw(\Rn)\to \Da(\Spp)$;
\item $W,V,U:\Sw'(\Rn)\to \Da'(\Spp)$.
\end{itemize}
\end{lemma}
\begin{proof}
The statements for $W$ are contained in \cite[Proposition 4.3]{HaPoRo18}, and in fact $W:L^{2}(\Rn)\to L^{2}(\Spp,\ud x\ud\w\frac{\ud\sigma}{\sigma})$ is an isometry. Next, let $f\in L^{2}(\Rn)$ and note that
\begin{align*}
&\|Vf\|_{L^{2}(\Spp)}^{2}=\int_{0}^{1}\int_{S^{n-1}}\|\theta_{\w,\sigma}(D)f\|_{L^{2}(\Rn)}^{2}\ud\w\frac{\ud\sigma}{\sigma}+\|s(D)f\|_{L^{2}(\Rn)}^{2}\\
&\lesssim (2\pi)^{-n}\int_{0}^{1}\int_{S^{n-1}}\int_{\Rn}|\theta_{\w,\sigma}(\zeta)\wh{f}(\zeta)|^{2}\ud\zeta\ud\w\frac{\ud\sigma}{\sigma}+\|f\|_{L^{2}(\Rn)}^{2}\\
&=(2\pi)^{-n}\int_{\Rn}\Big(\int_{0}^{1}\Psi_{\sigma}(\zeta)^{2}\frac{\ud\sigma}{\sigma}\Big)\Big(\int_{S^{n-1}}\ph_{\w}(\zeta)^{2}\ud\w\Big)|\wh{f}(\zeta)|^{2}\ud\zeta+\|f\|_{L^{2}(\Rn)}^{2}\\
&\lesssim \int_{\Rn}|\wh{f}(\zeta)|^{2}\ud\zeta+\|f\|_{L^{2}(\Rn)}^{2}\lesssim \|f\|_{L^{2}(\Rn)}^{2},
\end{align*}
where we used \eqref{eq:Psi}, \eqref{eq:phiomega1} and that $s\in C^{\infty}_{c}(\Rn)$. The proof that $U:L^{2}(\Rn)\to L^{2}(\Spp)$ is bounded is similar, using also \eqref{eq:phiomega2} with $\alpha=0$ and $\beta=0$.

The final two statements for $V$ and $U$ follow as in \cite[Proposition 4.3]{HaPoRo18}; the proofs only use the properties of the wave packets in Lemma \ref{lem:packetbounds} and that $s,h\in C^{\infty}_{c}(\Rn)$.
\end{proof}

This lemma shows that one can extend $W^{*}$, $V^{*}$ and $U^{*}$ to continuous maps from $\Da'(\Spp)$ to $\Sw'(\Rn)$, and the following reproducing formulas then hold:
\begin{equation}\label{eq:repro}
W^{*}Wf=U^{*}Vf=f\quad(f\in\Sw'(\Rn)).
\end{equation}
Indeed, $W^{*}Wf=f$ because $W:L^{2}(\Rn)\to L^{2}(\Spp,\ud x\ud\w\frac{\ud\sigma}{\sigma})$ is an isometry, and for $f\in\Sw(\Rn)$ and $\zeta\in\Rn$ one has
\begin{align*}
\F(U^{*}Vf)(\zeta)&=\int_{0}^{1}\int_{S^{n-1}}\chi_{\w,\sigma}(\zeta)\theta_{\w,\sigma}(\zeta)\wh{f}(\zeta)\ud\w\frac{\ud\sigma}{\sigma}+\int_{1}^{e}\fint_{S^{n-1}}h(\zeta)s(\zeta)\wh{f}(\zeta)\ud \w\frac{\ud\sigma}{\sigma}\\
&=\Big(\int_{0}^{1}\int_{S^{n-1}}\chi_{\w,\sigma}(\zeta)\theta_{\w,\sigma}(\zeta)\ud\w\frac{\ud\sigma}{\sigma}+s(\zeta)\Big)\wh{f}(\zeta)=\wh{f}(\zeta),
\end{align*}
by definition of $s$ in \eqref{eq:s}. For general $f\in\Sw'(\Rn)$ one then obtains $U^{*}Vf=f$ by approximation, using Lemma \ref{lem:transforms}.

We can now describe $\Hp$, for all $p\in[1,\infty]$, in terms of the wave packet transform $V$ and the tent space $T^{p}(\Sp)$ from Definition \ref{def:tent spaces}.

\begin{proposition}\label{prop:Hpnew}
Let $p\in[1,\infty]$. Then there exists a constant $C>0$ such that the following holds for all $f\in\Sw'(\Rn)$. One has $f\in\Hp$ if and only if $Vf\in T^{p}(\Sp)$, and then
\[
\frac{1}{C}\|f\|_{\Hp}\lesssim \|Vf\|_{T^{p}(\Sp)}\lesssim C\|f\|_{\Hp}.
\]
\end{proposition}
\begin{proof}
The proposition essentially follows by repeating the arguments in \cite{HaPoRo18} that show that $\Hp$ is independent of the choice of wave packets. More precisely, for $\sigma,\tau>0$, let $K_{\sigma,\tau}$ be the kernel on $\Sp$ of either $W_{\sigma}U_{\tau}^{*}$ or $V_{\sigma}W_{\tau}^{*}$. It suffices to show that, for each $N\geq0$, there exists a $C_{N}\geq0$ independent of $\sigma$ and $\tau$ such that 
\begin{equation}\label{eq:kernelbounds}
|K_{\sigma,\tau}((x,\w),(y,\nu))|\leq C_{N}\min(\tfrac{\sigma}{\tau},\tfrac{\tau}{\sigma})^{N}\rho^{-n}(1+\rho^{-1}d((x,\w),(y,\nu))^{2})^{-N}
\end{equation}
for all $(x,\w),(y,\nu)\in\Sp$, where $\rho:=\min(\sigma,\tau)$. Indeed, then \cite[Theorem 3.7]{HaPoRo18} shows that $WU^{*},VW^{*}\in \La(T^{p}(\Sp))$, and \eqref{eq:repro} yields
\begin{align*}
\|f\|_{\Hp}&=\|Wf\|_{T^{p}(\Sp)}=\|WU^{*}Vf\|_{T^{p}(\Sp)}\lesssim \|Vf\|_{T^{p}(\Sp)}\\
&=\|VW^{*}Wf\|_{T^{p}(\Sp)}\lesssim \|Wf\|_{T^{p}(\Sp)}=\|f\|_{\Hp}
\end{align*}
for all $f\in\Sw'(\Rn)$ such that one of these quantities is finite.

To obtain \eqref{eq:kernelbounds}, one can repeat the arguments in \cite[Section 5]{HaPoRo18}, which use only the properties of the wave packets in Lemma \ref{lem:packetbounds} and that $s,h\in C^{\infty}_{c}(\Rn)$. However, in this case the proof can be simplified considerably, and we will indicate how to do so here for the convenience of the reader.

Fix $\sigma,\tau>0$ and $(x,\w),(y,\nu)\in\Sp$, and first suppose that $\sigma,\tau<1$. We consider the case where $K_{\sigma,\tau}$ is the kernel of $W_{\sigma}U_{\tau}^{*}$, with the other case being almost identical. Note that
\begin{equation}\label{eq:kernelrep}
K_{\sigma,\tau}((x,\w),(y,\nu))=(2\pi)^{-n}\int_{\Rn}e^{i(x-y)\cdot\zeta}\psi_{\w,\sigma}(\zeta)\chi_{\nu,\tau}(\zeta)\ud\zeta.
\end{equation} 
If $\sigma\notin [\tau/4,4\tau]$, then the support properties of $\psi_{\w,\sigma}$ and $\chi_{\nu,\tau}$ from Lemma \ref{lem:packetbounds} imply that $K_{\sigma,\tau}=0$. So we assume henceforth that $\frac{\tau}{4}\leq \sigma\leq 4\tau$. For the same reason we may assume that $|\w-\nu|\leq 2(\sqrt{\sigma}+\sqrt{\tau})$. Then, by \eqref{eq:metric}, it suffices to show that 
\[
|K_{\sigma,\tau}((x,\w),(y,\nu))|\lesssim \sigma^{-n}\big(1+\sigma^{-1}(|x-y|^{2}+|\lb\w,x-y\rb|)\big)^{-N}.
\]
To this end we consider the differential operators
\[
D_{1}:=(1+\sigma^{-1}|x-y|^{2})^{-1}(1-\sigma^{-1}\Delta_{\zeta})
\]
and
\[
D_{2}:=(1+\sigma^{-2}|\lb\w,x-y\rb|^{2})^{-1}(1-\sigma^{-2}\lb\w,\nabla_{\zeta}\rb^{2}).
\] 
Note that $D_{1}e^{i(x-y)\cdot\zeta}=D_{2}e^{i(x-y)\cdot\zeta}=e^{i(x-y)\cdot\zeta}$, and that
\begin{align*}
&(1+\sigma^{-1}|x-y|^{2})^{-1}(1+\sigma^{-2}|\lb \w,x-y\rb|^{2})^{-1}\\
&\lesssim (1+\sigma^{-1}|x-y|^{2})^{-1}(1+\sigma^{-1}|\lb \w,x-y\rb|)^{-2}\\
&\leq \big(1+\sigma^{-1}(|x-y|^{2}+|\lb \w,x-y\rb|)\big)^{-1}.
\end{align*}
Now one integrates by parts repeatedly with respect to $D_{1}$ and $D_{2}$ in \eqref{eq:kernelrep} to introduce these weight factors, and then one uses the properties of $\psi_{\w,\sigma}$ and $\chi_{\nu,\tau}$ in Lemma \ref{lem:packetbounds} to bound $|K_{\sigma,\tau}((x,\w),(y,\nu))|$ from above by a multiple of
\begin{align*}
&\big(1+\sigma^{-1}(|x-y|^{2}+|\lb \w,x-y\rb|)\big)^{-N}\int_{\supp(\psi_{\w,\sigma})}\sigma^{-\frac{n-1}{4}}\tau^{-\frac{n-1}{4}}\ud\zeta\\
&\lesssim \sigma^{-n}\big(1+\sigma^{-1}(|x-y|^{2}+|\lb \w,x-y\rb|)\big)^{-N},
\end{align*}
as required.

The case where $\max(\sigma,\tau)\geq 1$ is similar but simpler. One may assume that $\sigma,\tau\in[c,e]$ for some $c>0$, and then it suffices to show that
\[
|K_{\sigma,\tau}((x,\w),(y,\nu))|\lesssim \big(1+|x-y|^{2}\big)^{-N},
\]
since $S^{n-1}$ is compact and $|\lb\w,x-y\rb|\leq 1+|x-y|^{2}$. For this in turn one can integrate by parts with respect to $(1+|x-y|^{2})(1-\Delta_{\zeta})$ in a similar representation for $K_{\sigma,\tau}$ as in \eqref{eq:kernelrep}.
\end{proof}

\begin{remark}\label{rem:Fmult}
Let $T$ be a \emph{normal oscillatory integral operator} of order zero, as in \cite[Definition 2.11]{HaPoRo18}, that commutes with the Laplacian. Examples of such operators are the wave operators $e^{it\sqrt{-\Delta}}$ for $t\in\R$, as well as $m(D)$ for $m\in C^{\infty}(\Rn)$ satisfying standard symbol estimates of order zero. Then the arguments in the proof of Proposition \ref{prop:Hpnew} can easily be modified to obtain the following version of \eqref{eq:kernelbounds} for the kernel $K_{\sigma,\tau}$ of either $W_{\sigma}TU_{\tau}$ or $V_{\sigma}TW^{*}_{\tau}$:
\begin{equation}\label{eq:offsing}
|K_{\sigma,\tau}((x,\w),(y,\nu))|\leq C_{N}\min(\tfrac{\sigma}{\tau},\tfrac{\tau}{\sigma})^{N}\rho^{-n}(1+\rho^{-1}d((x,\w),\hat{\chi}(y,\nu))^{2})^{-N},
\end{equation}
for $N\geq0$, $(x,\w,\sigma),(y,\nu,\tau)\in\Spp$, $\rho=\min(\sigma,\tau)$ and a suitable bi-Lipschitz map $\hat{\chi}:\Sp\to\Sp$ associated with $T$. In \cite{HaPoRo18} such bounds were called \emph{off-singularity bounds}, and they play a key role in showing that $T$ is a bounded operator on $\Hp$ for all $p\in[1,\infty]$. In fact, \eqref{eq:offsing} and \cite[Theorem 3.7]{HaPoRo18} yield
\begin{equation}\label{eq:tentbdd}
WTU\in \La(T^{p}(\Sp)),
\end{equation} 
and it then follows from \eqref{eq:repro} and Proposition \ref{prop:Hpnew} that 
\[
\|Tf\|_{\Hp}=\|WTUVf\|_{T^{p}(\Sp)}\lesssim \|Vf\|_{T^{p}(\Sp)}\eqsim \|f\|_{\Hp}
\]
for all $f\in\Hp$. In \cite[Section 5]{HaPoRo18} the off-singularity bounds in \eqref{eq:offsing} were obtained for more general normal oscillatory integral operators, but with a considerably more involved proof that can be simplified if $T$ commutes with the Laplacian. 
\end{remark}

As in \cite[Corollary 7.6]{HaPoRo18}, we can use Proposition \ref{prop:Hpnew} to obtain equivalent norms on $\Hp$ for all $p\in[1,\infty]$. For $f\in\Sw'(\Rn)$ and $(x,\w)\in\Sp$, set
\[
Sf(x,\w):=\Big(\int_{0}^{1}\fint_{B_{\sqrt{\sigma}}(x,\w)}|\theta_{\nu,\sigma}(D)f(y)|^{2}\ud y\ud\nu\frac{\ud\sigma}{\sigma}\Big)^{1/2}\in[0,\infty]
\]
and
\[
Q f(x,\w):=\sup_{B}\Big(\frac{1}{V(B)}\int_{T(B)}\ind_{[0,1]}(\sigma)|\theta_{\nu,\sigma}(D)f(y)|^{2}\ud y\ud\nu\frac{\ud\sigma}{\sigma}\Big)^{1/2}\in[0,\infty],
\]
where the supremum is taken over all balls $B\subseteq\Sp$ containing $(x,\w)$. 

\begin{corollary}\label{cor:equivnorm}
Let $q\in C^{\infty}_{c}(\Rn)$ be such that $q(\zeta)=1$ for $|\zeta|\leq 2$, and let $p\in[1,\infty]$ and $f\in\Sw'(\Rn)$. Then the following assertions hold.
\begin{enumerate}
\item For $p<\infty$, one has $f\in\Hp$ if and only if $Sf\in L^{p}(\Sp)$ and $q(D)f\in L^{p}(\Rn)$, and then 
\[
\|f\|_{\Hp}\eqsim \|Sf\|_{L^{p}(\Sp)}+\|q(D)f\|_{L^{p}(\Rn)}
\]
for an implicit constant independent of $f$. 
\item One has $f\in \HT^{\infty}_{FIO}(\Rn)$ if and only if $Qf\in L^{\infty}(\Sp)$ and $q(D)f\in L^{\infty}(\Rn)$, and then 
\[
\|f\|_{\HT^{\infty}_{FIO}(\Rn)}\eqsim \|Qf\|_{L^{p}(\Sp)}+\|q(D)f\|_{L^{\infty}(\Rn)}
\]
for an implicit constant independent of $f$.
\end{enumerate}
\end{corollary}
\begin{proof}
We argue as in the proof of \cite[Corollary 7.6]{HaPoRo18}. For $(x,\w,\sigma)\in\Spp$ let
\[
F(x,\w,\sigma):=\begin{cases}
\theta_{\w,\sigma}(D)f(x)&\text{if }\sigma<1,\\
0&\text{otherwise},
\end{cases}
\]
and first consider $p<\infty$. Then $\|F\|_{T^{p}(\Sp)}=\|Sf\|_{L^{p}(\Sp)}$ and $(1-s)(D)f=U^{*}F$. Moreover, $1-q=(1-q)(1-s)$. Now, for $s>0$ sufficiently large, we use the Sobolev embeddings for $\Hp$ in \cite[Theorem 7.4]{HaPoRo18}, as well as \eqref{eq:tentbdd} with $T=(1-q)(D)$ and that $q\in C^{\infty}_{c}(\Rn)$, to write
\begin{align*}
\|f\|_{\Hp}&\leq \|(1-q)(D)f\|_{\Hp}+\|q(D)f\|_{\Hp}\\
&\lesssim \|(1-q)(D)(1-s)(D)f\|_{\Hp}+\|q(D)f\|_{W^{s,p}(\Rn)}\\
&\lesssim \|W(1-q)(D)U^{*}F\|_{T^{p}(\Sp)}+\|q(D)f\|_{L^{p}(\Rn)}\\
&\lesssim \|F\|_{T^{p}(\Sp)}+\|q(D)f\|_{L^{p}(\Rn)}=\|Sf\|_{L^{p}(\Sp)}+\|q(D)f\|_{L^{p}(\Rn)}\\
&\lesssim \|Vf\|_{T^{p}(\Sp)}+\|f\|_{W^{-s,p}(\Rn)}\lesssim \|f\|_{\Hp}.
\end{align*}
For $p=\infty$ one simply replaces $Sf$ by $Qf$ throughout.
\end{proof}

\subsection{Change of wave packets}\label{subsec:changes}

In the next section we will want to move between different choices of wave packets, for which we need a lemma. Suppose that, in addition to the assumptions on $\ph$ and $\Psi$ from before, one has $\ph(\zeta)=0$ for all $|\zeta|>\frac{1}{4}$, and $\Psi(\zeta)=0$ if $|\zeta|\notin [\frac{4}{5},\frac{5}{4}]$. Let $\wt{\ph}\in C^{\infty}_{c}(\Rn)$ be a non-negative function such that $\wt{\ph}(\zeta)=1$ for $|\zeta|\leq \frac{3}{4}$ and $\wt{\ph}(\zeta)=0$ for $|\zeta|>1$. Let $\wt{\Psi}\in C_{c}^{\infty}(\Rn)$ be a non-negative radial function such that $\wt{\Psi}(\zeta)=1$ for $|\zeta|\in[\frac{2}{3},\frac{3}{2}]$, $\wt{\Psi}(\zeta)=0$ for $|\zeta|\notin [\frac{1}{2},2]$, and 
\[
\int_{0}^{\infty}\wt{\Psi}(\sigma\zeta)^{2}\frac{\ud \sigma}{\sigma}=1\quad(\zeta\neq0).
\]
Note that such a $\wt{\Psi}$ exists because $\int_{2/3}^{3/2}\frac{\ud\sigma}{\sigma}=\log(9/4)<1$. In the same manner as before, for $\w\in S^{n-1}$, $\sigma>0$ and $\zeta\in\Rn\setminus\{0\}$, set $\wt{\ph}_{\w,\sigma}(\zeta):=\wt{c}_{\sigma}\wt{\ph}\big(\tfrac{\hat{\zeta}-\w}{\sqrt{\sigma}}\big)$, where $\wt{c}_{\sigma}:=\big(\int_{S^{n-1}}\wt{\ph}\big(\tfrac{e_{1}-\nu}{\sqrt{\sigma}}\big)^{2}\ud\nu\big)^{-1/2}$. Let $\wt{\ph}_{\w,\sigma}(0):=0$, and set $\wt{\Psi}_{\sigma}(\zeta):=\wt{\Psi}(\sigma\zeta)$, $\wt{\psi}_{\w,\sigma}(\zeta):=\wt{\Psi}_{\sigma}(\zeta)\wt{\ph}_{\w,\sigma}(\zeta)$, and
\[
\wt{\ph}_{\w}(\zeta):=\int_{0}^{4}\wt{\psi}_{\w,\tau}(\zeta)\frac{\ud\tau}{\tau},\quad \wt{\theta}_{\w,\sigma}(\zeta):=\wt{\Psi}_{\sigma}(\zeta)\wt{\ph}_{\w}(\zeta).
\]
It then follows from Proposition \ref{prop:Hpnew}, with $\ph$ and $\Psi$ replaced by $\wt{\ph}$ and $\wt{\Psi}$, that one can use these new wave packets to obtain an equivalent norm on $\Hp$ for all $p\in[1,\infty]$. In particular, 
\begin{equation}\label{eq:newnorm}
\Big(\int_{\Sp}\Big(\int_{0}^{1}\fint_{B_{\sqrt{\sigma}}(x,\w)}|\wt{\theta}_{\w,\sigma}(D)f(y)|^{2}\ud y\ud\nu\frac{\ud\sigma}{\sigma}\Big)^{\frac{p}{2}}\ud x\ud\w\Big)^{\frac{1}{p}}\lesssim \|f\|_{\Hp}
\end{equation}
for all $p<\infty$ and $f\in\Hp$.

\begin{lemma}\label{lem:vary}
There exists a constant $c>0$ such that the following holds. Let $\w,\nu \in S^{n-1}$ and $\sigma\in(0,1)$ be such that $|\w-\nu|\leq \frac{1}{16}\sqrt{\sigma}$. Define $\eta_{\w,\nu,\sigma}:\Rn\to[0,\infty)$ by 
\[
\eta_{\w,\nu,\sigma}(\zeta):=\begin{cases}
\frac{\theta_{\w,\sigma}(\zeta)}{\tilde{\theta}_{\nu,\sigma}(\zeta)}&\text{for }\zeta\in\supp(\theta_{\w,\sigma}),\\
0&\text{otherwise}.
\end{cases}
\]
Then $\eta_{\w,\nu,\sigma}\in C^{\infty}_{c}(\Rn)$, and for each $N\geq0$ there exists a $C_{N}\geq0$, independent of $\w$, $\nu$ and $\sigma$, such that
\[
|\F^{-1}(\eta_{\w,\nu,\sigma})(x)|\leq C_{N}\sigma^{-\frac{n+1}{2}}(1+\sigma^{-1}|x|^{2}+\sigma^{-2}\lb\w,x\rb^{2})^{-N}\quad(x\in \Rn).
\]
\end{lemma}
\begin{proof}
We first show that $\wt{\theta}_{\nu,\sigma}$ is uniformly bounded away from zero on $\supp(\theta_{\w,\sigma})$. Note that, by the support condition on $\Psi$, one has $\Psi_{\sigma}\Psi_{\tau}=0$ if $\tau\notin[(\frac{4}{5})^{2}\sigma,(\frac{5}{4})^{2}\sigma]$. Hence
\[
\theta_{\w,\sigma}(\zeta)=\Psi_{\sigma}(\zeta)\int_{0}^{4}\Psi_{\tau}(\zeta)\ph_{\w,\tau}(\zeta)\frac{\ud\tau}{\tau}=\int_{(\frac{4}{5})^{2}\sigma}^{(\frac{5}{4})^{2}\sigma}\Psi_{\sigma}(\zeta)\Psi_{\tau}(\zeta)c_{\tau}\ph\big(\tfrac{\hat{\zeta}-\w}{\sqrt{\tau}}\big)\frac{\ud\tau}{\tau}
\]
for all $\zeta\in\Rn$, where we also used that $(\frac{5}{4})^{2}\sigma<4$. So if $\theta_{\w,\sigma}(\zeta)\neq 0$ then $\ph\big(\tfrac{\hat{\zeta}-\w}{\sqrt{\tau}}\big)\neq 0$ for some $\tau\leq (\tfrac{5}{4})^{2}\sigma$. By the support condition on $\ph$, this implies that $|\hat{\zeta}-\w|\leq \frac{1}{4}\sqrt{\tau}\leq \frac{5}{16}\sqrt{\sigma}$. Since $|\w-\nu|\leq\frac{1}{16}\sqrt{\sigma}$, this in turn yields
\[
|\hat{\zeta}-\nu|\leq |\hat{\zeta}-\w|+|\w-\nu|\leq \tfrac{3}{8}\sqrt{\sigma}\quad(\zeta\in\supp(\theta_{\w,\sigma})).
\]
Hence $\frac{|\hat{\zeta}-\nu|}{\sqrt{\tau}}\leq2\frac{|\hat{\zeta}-\nu|}{\sqrt{\sigma}}\leq \frac{3}{4}$ for all $\tau\geq\frac{1}{4}\sigma$, and the properties of $\wt{\ph}$ now imply that
\[
\wt{\ph}\big(\tfrac{\hat{\zeta}-\nu}{\sqrt{\tau}}\big)=1\quad(\zeta\in\supp(\theta_{\w,\sigma}),\tau\geq \tfrac{1}{4}\sigma).
\]
Next, note that $\wt{\Psi}_{\sigma}\wt{\Psi}_{\tau}=0$ if $\tau\notin[\frac{1}{4}\sigma,4\sigma]$, so that
\begin{align*}
\wt{\theta}_{\nu,\sigma}(\zeta)&=\wt{\Psi}_{\sigma}(\zeta)\int_{0}^{4}\wt{\Psi}_{\tau}(\zeta)\wt{\ph}_{\nu,\tau}(\zeta)\frac{\ud\tau}{\tau}=\int_{\frac{1}{4}\sigma}^{4\sigma}\wt{\Psi}_{\sigma}(\zeta)\wt{\Psi}_{\tau}(\zeta)\wt{c}_{\tau}\wt{\ph}\big(\tfrac{\hat{\zeta}-\nu}{\sqrt{\tau}}\big)\frac{\ud\tau}{\tau}\\
&\geq \int_{\frac{5}{6}\sigma}^{\frac{6}{5}\sigma}\wt{\Psi}_{\sigma}(\zeta)\wt{\Psi}_{\tau}(\zeta)\wt{c}_{\tau}\frac{\ud\tau}{\tau},
\end{align*}
where we also used that $\wt{\Psi}$ is positive. Moreover, for all $\tau\in[\frac{5}{6}\sigma,\frac{6}{5}\sigma]$ and $\zeta\in\supp(\theta_{\w,\sigma})\subseteq\supp(\Psi_{\sigma})$ one has
\[
\tfrac{2}{3}\tau^{-1}\leq \tfrac{4}{5}\sigma^{-1}\leq |\zeta|\leq \tfrac{5}{4}\sigma^{-1}\leq \tfrac{3}{2}\tau^{-1},
\]
so that $\wt{\Psi}_{\tau}(\zeta)=1$. Finally, by \eqref{eq:csigma} with $\ph$ replaced by $\wt{\ph}$, one has
\[
\wt{c}_{\tau}\eqsim \tau^{-\frac{n-1}{4}}.
\]
By combining all this, we find
\begin{equation}\label{eq:thetabound}
\wt{\theta}_{\nu,\sigma}(\zeta)\geq \int_{\frac{5}{6}\sigma}^{\frac{6}{5}\sigma}\wt{\Psi}_{\sigma}(\zeta)\wt{\Psi}_{\tau}(\zeta)\wt{c}_{\tau}\frac{\ud\tau}{\tau}\eqsim \int_{\frac{5}{6}\sigma}^{\frac{6}{5}\sigma}\tau^{-\frac{n-1}{4}}\frac{\ud\tau}{\tau}\eqsim \sigma^{-\frac{n-1}{4}}
\end{equation}
for implicit constants independent of $\w$, $\nu$, $\sigma$ and $\zeta\in\supp(\theta_{\w,\sigma})$. This shows in particular that $\eta_{\w,\nu,\sigma}$ is well defined.

Now, by Lemma \ref{lem:packetbounds} one has
\[
|\lb\w,\nabla_{\zeta}\rb^{\beta}\partial_{\zeta}^{\alpha}\theta_{\w,\sigma}(\zeta)|\lesssim \sigma^{-\frac{n-1}{4}+\frac{|\alpha|}{2}+\beta}
\]
for all $\alpha\in\Z_{+}^{n}$ and $\beta\in\Z_{+}$, and the same bounds hold for $\wt{\theta}_{\nu,\sigma}$ since $|\w-\nu|\leq \frac{1}{16}\sqrt{\sigma}$. We can now combine this with \eqref{eq:thetabound} to show that $\eta_{\w,\nu,\sigma}\in C^{\infty}_{c}(\Rn)$ with 
\[
|\lb\w,\nabla_{\zeta}\rb^{\beta}\partial_{\zeta}^{\alpha}\eta_{\w,\nu,\sigma}(\zeta)|\lesssim \sigma^{\frac{|\alpha|}{2}+\beta}
\]
for an implicit constant independent of $\w$, $\nu$ and $\sigma$. To conclude the proof, one can now integrate by parts as in the proof of Lemma \ref{lem:packetbounds} (see also \cite[Lemma 4.1]{HaPoRo18}), using the support properties of $\theta_{\w,\sigma}$.
\end{proof}

\section{Characterizations of Hardy spaces for FIOs}\label{sec:char}

In this section we prove our main result, Theorem \ref{thm:intro}, which characterizes $\Hp$ for $1<p<\infty$ in terms of $L^{p}(\Sp)$-norms and Fourier multipliers that localize to paraboloids in frequency. Then, using known characterizations of $L^{p}(\Rn)$, we can derive various alternative characterizations of $\Hp$ as corollaries.

\subsection{The main theorem}

Let $\ph\in C^{\infty}_{c}(\Rn)$ be a non-negative function such that $\phi\equiv1$ in a neighborhood of zero, and $\ph(\zeta)=0$ for $|\zeta|>\frac{1}{4}$. Let $\Psi\in C^{\infty}_{c}(\Rn)$ be a non-negative radial function such that $\Psi(\zeta)=0$ if $|\zeta|\notin[\frac{4}{5},\frac{5}{4}]$, and
\[
\int_{0}^{\infty}\Psi(\sigma\zeta)^{2}\frac{\ud \sigma}{\sigma}=1\quad(\zeta\neq0).
\]
For $\w\in S^{n-1}$, define $\ph_{\w}\in C^{\infty}(\Rn)$ as in \eqref{eq:phiomega}:
\[
\ph_{\w}(\zeta):=\int_{0}^{4}\psi_{\w,\tau}(\zeta)\frac{\ud\tau}{\tau}\quad(\zeta\in\Rn),
\]
where $\psi_{\w,\tau}$ is as defined in Section \ref{sec:wavepackets} for $\tau>0$. As noted in Remark \ref{rem:phifunctions}, $\ph_{\w}$ grows polynomially at infinity, and therefore $\ph_{\w}(D)f\in \Sw'(\Rn)$ is well defined for all $f\in\Sw'(\Rn)$. Also, as indicated in Remark \ref{rem:support}, with minor modifications one could allow for more general support assumptions on $\ph$ and $\Psi$ in what follows. Let $q\in C^{\infty}_{c}(\Rn)$ be such that $q(\zeta)=1$ if $|\zeta|\leq 2$.

The following theorem is our main result, already stated in the introduction as Theorem \ref{thm:intro}.

\begin{theorem}\label{thm:chargeneral}
Let $p\in(1,\infty)$. Then there exists a constant $C>0$ such that the following holds for all $f\in\Sw'(\Rn)$. One has $f\in\Hp$ if and only if $q(D)f\in L^{p}(\Rn)$, $\ph_{\w}(D)f\in L^{p}(\Rn)$ for almost all $\w\in S^{n-1}$, and 
\[
\Big(\int_{\Sp}|\ph_{\w}(D)f(x)|^{p}\ud x\ud\w\Big)^{\frac{1}{p}}<\infty.
\]
Moreover, 
\[
\frac{1}{C}\|f\|_{\Hp}\leq \|q(D)f\|_{L^{p}(\Rn)}+\Big(\int_{\Sp}|\ph_{\w}(D)f(x)|^{p}\ud x\ud\w\Big)^{\frac{1}{p}}\leq C\|f\|_{\Hp}
\]
if $f\in\Hp$.
\end{theorem}
\begin{proof}
We first prove that
\begin{equation}\label{eq:righthand}
\|q(D)f\|_{L^{p}(\Rn)}+\Big(\int_{\Sp}|\ph_{\w}(D)f(x)|^{p}\ud x\ud\w\Big)^{\frac{1}{p}}\lesssim \|f\|_{\Hp}
\end{equation}
if $f\in\Hp$, after which we use duality to derive the remaining statements.

Suppose $f\in\Hp$. We first deal with the low frequencies of $f$. For any $m\in C^{\infty}_{c}(\Rn)$ one has $m(D)f\in L^{p}(\Rn)$ and
\begin{equation}\label{eq:lowfreq}
\|m(D)f\|_{L^{p}(\Rn)}\lesssim \|f\|_{\Hp},
\end{equation}
with an implicit constant which depends only on $\supp(m)$ and on $\|\F^{-1}(m)\|_{L^{1}(\Rn)}$. This follows e.g.~from the Sobolev embeddings for $\Hp$ in \cite[Theorem 7.4]{HaPoRo18}:
\[
\|m(D)f\|_{L^{p}(\Rn)}\eqsim \|\lb D\rb^{s_{p}}m(D)f\|_{W^{-s_{p},p}(\Rn)}\lesssim \|f\|_{W^{-s_{p},p}(\Rn)}\lesssim \|f\|_{\Hp},
\]
where $s_{p}=\frac{n-1}{2}|\frac{1}{p}-\frac{1}{2}|$ and where we used that $[\zeta\mapsto \lb \zeta\rb^{\frac{n-1}{2}}m(\zeta)]\in C^{\infty}_{c}(\Rn)$. Applying \eqref{eq:lowfreq} with $m=q$, we obtain that $q(D)f\in L^{p}(\Rn)$ and
\[
\|q(D)f\|_{L^{p}(\Rn)}\lesssim \|f\|_{\Hp}.
\]
Next, for $\w\in S^{n-1}$ we use \eqref{eq:lowfreq} with $m=\ph_{\w}q$:
\[
\|\ph_{\w}(D)q(D)f\|_{L^{p}(\Rn)}\lesssim \|f\|_{\Hp},
\]
where the implicit constant is independent of $\w$ due to the bounds for $\ph_{\w}$ in Remark \ref{rem:phifunctions}. Hence we obtain
\begin{align*}
&\|q(D)f\|_{L^{p}(\Rn)}+\Big(\int_{S^{*}(\Rn)}|\ph_{\w}(D)q(D)f(x)|^{p}\ud x\ud\w\Big)^{1/p}\\
&\lesssim\|q(D)f\|_{L^{p}(\Rn)}+\sup_{\w\in S^{n-1}}\|\ph_{\w}(D)q(D)f\|_{L^{p}(\Rn))}\lesssim\|f\|_{\Hp}.
\end{align*}

Now, for \eqref{eq:righthand} it remains to consider the high frequencies of $f$, captured by $(1-q)(D)f$. We have to show that
\[
\Big(\int_{S^{n-1}}\|\ph_{\w}(D)(1-q)(D)f\|_{L^{p}(\Rn)}^{p}\ud\w\Big)^{1/p}\lesssim\|f\|_{\Hp}.
\]
To prove this we will in turn use the conical square function characterization of $L^{p}(\Rn)$ (see \cite[Section I.8.C]{Stein93}): for $g\in\Sw'(\Rn)$ one has $g\in L^{p}(\Rn)$ if and only if
\[
\Big(\int_{\Rn}\Big(\int_{0}^{\infty}\fint_{B_{\sigma}(x)}|\Psi_{\sigma}(D)g(y)|^{2}\ud y\frac{\ud\sigma}{\sigma}\Big)^{p/2}\ud x\Big)^{1/p}<\infty,
\]
in which case $\|g\|_{L^{p}(\Rn)}$ is comparable to this quantity. Here $B_{\sigma}(x)\subseteq\Rn$ is the Euclidean ball of radius $\sigma>0$ around $x\in\Rn$. Using this characterization with $g=\ph_{\w}(D)\wt{f}$ for $\wt{f}:=(1-q)(D)f$ and $w\in S^{n-1}$, we see that for \eqref{eq:righthand} it suffices to show that
\begin{equation}\label{eq:1-qterm}
\Big(\int_{\Sp}\Big(\int_{0}^{\infty}\fint_{B_{\sigma}(x)}|\Psi_{\sigma}(D)\ph_{\w}(D)\wt{f}(y)|^{2}\ud y\frac{\ud\sigma}{\sigma}\Big)^{\frac{p}{2}}\ud x\ud\w\Big)^{\frac{1}{p}}\lesssim \|f\|_{\Hp}.
\end{equation}
To prove \eqref{eq:1-qterm} we have to replace the averages over the isotropic balls $B_{\sigma}(x)$ in $\Rn$ by averages over anisotropic balls $B_{\sqrt{\sigma}}(x,\w)$ in $\Sp$. To do so, we will make a suitable change of wave packets, and then we decompose $\Rn$ into anisotropic annuli on which the inverse Fourier transforms of these new wave packets decay rapidly. 

Fix $(x,\w)\in \Sp$. We first do some preliminary work. For $\nu\in B_{\sqrt{\sigma}/16}(\w)$ and $\sigma\in(0,1)$, set 
\[
\eta_{\w,\nu,\sigma}(\zeta):=\begin{cases}
\frac{\theta_{\w,\sigma}(\zeta)}{\tilde{\theta}_{\nu,\sigma}(\zeta)}&\text{for }\zeta\in\supp(\theta_{\w,\sigma}),\\
0&\text{otherwise}.
\end{cases}
\]
By Lemma \ref{lem:vary}, one has $\eta_{\w,\nu,\sigma}\in C^{\infty}_{c}(\Rn)$. Next, let 
\begin{equation}\label{eq:jzero}
C_{0,\sigma}:=\{z\in\Rn\mid |z|^{2}+|\lb \w,z\rb|\leq \sigma\} 
\end{equation}
and, for $j\in\N$,  
\[
C_{j,\sigma}:=\{z\in \Rn\mid 2^{j-1}\sigma<|z|^{2}+|\lb \w,z\rb|\leq 2^{j}\sigma\}.
\]
It follows from geometric considerations that
\begin{equation}\label{eq:volume1}
|C_{j,\sigma}|\eqsim (2^{j}\sigma)^{\frac{n+1}{2}}
\end{equation}
for all $j\in\Z_{+}$ such that $2^{j}\sigma\leq1$, and 
\begin{equation}\label{eq:volume2}
|C_{j,\sigma}|\eqsim (2^{j}\sigma)^{\frac{n}{2}}
\end{equation}
if $2^{j}\sigma>1$, for implicit constants independent of $\w$, $\sigma$ and $j$. Indeed, \eqref{eq:volume1} follows for $j\in\N$ by noting, for example, that
\begin{align*}
&\{z\in\Rn\mid 2^{j-1}\sigma\leq |\lb\w,z\rb|\leq \tfrac{6}{5}2^{j-1}\sigma,|z-\lb\w,z\rb|\leq \tfrac{1}{5}(2^{j}\sigma)^{\frac{1}{2}}\}\\
&\subseteq C_{j,\sigma}\subseteq\{z\in\Rn\mid |\lb\w,z\rb|\leq 2^{j}\sigma,|z-\lb\w,z\rb|\leq (2^{j}\sigma)^{\frac{1}{2}}\},
\end{align*}
and similarly for $j=0$ and $2^{j}\sigma>1$.

Now, to bound the left-hand side of \eqref{eq:1-qterm}, first note that for $\sigma\geq 1$ one has $\Psi_{\sigma}(D)\wt{f}=\Psi_{\sigma}(D)(1-q)(D)f=0$. Hence 
\begin{align*}
&\Big(\int_{0}^{\infty}\fint_{B_{\sigma}(x)}|\Psi_{\sigma}(D)\ph_{\w}(D)\wt{f}(y)|^{2}\ud y\frac{\ud\sigma}{\sigma}\Big)^{\frac{1}{2}}=\Big(\int_{0}^{\infty}\fint_{B_{\sigma}(x)}|\theta_{\w,\sigma}(D)\wt{f}(y)|^{2}\ud y\frac{\ud\sigma}{\sigma}\Big)^{\frac{1}{2}}\\
&=\Big(\int_{0}^{1}\fint_{B_{\sigma}(x)}\fint_{B_{\sqrt{\sigma}/16}(\w)}|\eta_{\w,\nu,\sigma}(D)\wt{\theta}_{\nu,\sigma}(D)\wt{f}(y)|^{2}\ud\nu\ud y\frac{\ud\sigma}{\sigma}\Big)^{\frac{1}{2}}\\
&=\Big(\int_{0}^{1}\fint_{B_{\sqrt{\sigma}/16}(\w)}\fint_{B_{\sigma}(x)}\Big|\sum_{j=0}^{\infty}\int_{C_{j,\sigma}}\F^{-1}(\eta_{\w,\nu,\sigma})(z)\wt{\theta}_{\nu,\sigma}(D)\wt{f}(y-z)\ud z\Big|^{2}\ud y\ud\nu\frac{\ud\sigma}{\sigma}\Big)^{\frac{1}{2}}\\
&\leq \sum_{j=0}^{\infty}\!\Big(\int_{0}^{1}\fint_{B_{\sqrt{\sigma}/16}(\w)}\fint_{B_{\sigma}(x)}\!\Big(\!\int_{C_{j,\sigma}}\!|\F^{-1}(\eta_{\w,\nu,\sigma})(z)\wt{\theta}_{\nu,\sigma}(D)\wt{f}(y-z)|\ud z\!\Big)^{2}\ud y\ud\nu\frac{\ud\sigma}{\sigma}\Big)^{\frac{1}{2}}\!.
\end{align*}
Next, we will bound each of the terms in this series separately. 

Fix $j\in\Z_{+}$ and note that for each $N\geq0$, Lemma \ref{lem:vary}, \eqref{eq:volume1} and \eqref{eq:volume2} yield
\begin{align*}
|\F^{-1}(\eta_{\w,\nu,\sigma})(z)|&\lesssim \sigma^{-\frac{n+1}{2}}\big(1+\sigma^{-1}|z|^{2}+\sigma^{-2}\lb\w,z\rb^{2}\big)^{-2(N+\frac{n+1}{2})}\\
&\lesssim \sigma^{-\frac{n+1}{2}}\big(1+\sigma^{-1}(|z|^{2}+|\lb\w,z\rb|)\big)^{-(N+\frac{n+1}{2})}\\
&\lesssim 2^{-jN}(2^{j}\sigma)^{-\frac{n+1}{2}}\lesssim 2^{-jN}|C_{j,\sigma}|^{-1}
\end{align*}
for all $\sigma\in(0,1)$ and $z\in C_{j,\sigma}$, where we also used that
\begin{align*}
&\big(1+\sigma^{-1}|z|^{2}+\sigma^{-2}\lb\w,z\rb^{2}\big)^{-2}\leq \big(1+\sigma^{-1}|z|^{2}\big)^{-1}\big(1+\sigma^{-2}\lb\w,z\rb^{2}\big)^{-1}\\
&\lesssim \big(1+\sigma^{-1}|z|^{2}\big)^{-1}\big(1+\sigma^{-1}|\lb\w,z\rb|\big)^{-2}\leq \big(1+\sigma^{-1}(|z|^{2}+|\lb\w,z\rb|)\big)^{-1}.
\end{align*}
For later use we fix $N>\frac{n-1}{4}+\frac{n}{2}$. Now the Cauchy-Schwarz inequality yields
\begin{equation}\label{eq:bigone}
\begin{aligned}
&\int_{0}^{1}\fint_{B_{\sqrt{\sigma}/16}(\w)}\fint_{B_{\sigma}(x)}\Big(\int_{C_{j,\sigma}}|\F^{-1}(\eta_{\w,\nu,\sigma})(z)\wt{\theta}_{\nu,\sigma}(D)\wt{f}(y-z)|\ud z\Big)^{2}\ud y\ud\nu\frac{\ud\sigma}{\sigma}\\
&\lesssim 2^{-2jN}\int_{0}^{1}\fint_{B_{\sqrt{\sigma}/16}(\w)}\fint_{B_{\sigma}(x)}\Big(\fint_{C_{j,\sigma}}|\wt{\theta}_{\nu,\sigma}(D)\wt{f}(y-z)|\ud z\Big)^{2}\ud y\ud\nu\frac{\ud\sigma}{\sigma}\\
&\lesssim 2^{-2jN}\int_{0}^{1}\fint_{B_{\sqrt{\sigma}/16}(\w)}\fint_{B_{\sigma}(x)}\fint_{y-C_{j,\sigma}}|\wt{\theta}_{\nu,\sigma}(D)\wt{f}(z)|^{2}\ud z\ud y\ud\nu\frac{\ud\sigma}{\sigma}.
\end{aligned}
\end{equation}
At this point we are able to remove the average over the isotropic ball $B_{\sigma}(x)$, and instead work with an average over a larger anisotropic ball. In fact, we will replace the integrals over $B_{\sqrt{\sigma}/16}(\w)$, $B_{\sigma}(x)$ and $y-C_{j,\sigma}$ by a single integral over a ball $B_{c2^{j/2}\sqrt{\sigma}}(x,\w)\subseteq\Sp$. This will come at the cost of a factor $2^{j\frac{n-1}{2}}$, but such a factor does not pose a problem given our choice of $N$.

Note that for all $\sigma\in(0,1)$ and $y,z\in\Rn$ such that $y\in B_{\sigma}(x)$ and $z\in y-C_{j,\sigma}$, one has
\begin{align*}
&|\lb \w,x-z\rb|+|x-z|^{2}\leq |\lb\w,x-y\rb|+|\lb\w,y-z\rb|+(|x-y|+|y-z|)^{2}\\
&\leq \sigma+|\lb\w,y-z\rb|+|y-z|^{2}+\sigma^{2}+2\sigma|y-z|\leq \sigma+2^{j}\sigma+\sigma+2\sigma2^{j/2}\sqrt{\sigma}\leq 5\cdot 2^{j}\sigma.
\end{align*}
Hence $z-x\in C_{0,2^{j}5\sigma}$ and $z\in x+C_{0,2^{j}5\sigma}$, where $C_{0,2^{j}5\sigma}$ is as in \eqref{eq:jzero} with $\sigma$ replaced by $2^{j}5\sigma$. It is straightforward to check, using \eqref{eq:volume1} and \eqref{eq:volume2}, that
\begin{equation}\label{eq:volume3}
|C_{0,2^{j}5\sigma}|\eqsim |C_{j,\sigma}|
\end{equation}
with implicit constants independent of $\w$, $\sigma$ and $j$. It now follows that
\begin{equation}\label{eq:removal}
\begin{aligned}
&\fint_{B_{\sigma}(x)}\fint_{y-C_{j,\sigma}}|\wt{\theta}_{\nu,\sigma}(D)\wt{f}(z)|^{2}\ud z\ud y\\
&=\int_{\Rn}\int_{\Rn}\ind_{B_{\sigma}(x)}(y)\ind_{y-C_{j,\sigma}}(z)|\wt{\theta}_{\nu,\sigma}(D)\wt{f}(z)|^{2}\frac{\ud z}{|C_{j,\sigma}|}\frac{\ud y}{|B_{\sigma}(x)|}\\
&\leq \int_{\Rn}\int_{\Rn}\ind_{B_{\sigma}(x)}(y)\ind_{x+C_{0,2^{j}5\sigma}}(z)|\wt{\theta}_{\nu,\sigma}(D)\wt{f}(z)|^{2}\frac{\ud y}{|B_{\sigma}(x)|}\frac{\ud z}{|C_{j,\sigma}|}\\
&\eqsim \int_{\Rn}\ind_{x+C_{0,2^{j}5\sigma}}(z)|\wt{\theta}_{\nu,\sigma}(D)\wt{f}(z)|^{2}\frac{\ud z}{|C_{0,2^{j}5\sigma}|}.
\end{aligned}
\end{equation}
Moreover, by \eqref{eq:metric}, there exists a constant $c=c(n)\geq1$ such that 
\[
(x+C_{0,2^{j}5\sigma})\times B_{\sqrt{\sigma}/16}(\w)\subseteq B_{c2^{j/2}\sqrt{\sigma}}(x,\w).
\]
Combining \eqref{eq:metric}, \eqref{eq:volume1}, \eqref{eq:volume2} and \eqref{eq:volume3}, we also obtain
\[
|B_{c2^{j/2}\sqrt{\sigma}}(x,\w)|\lesssim 2^{j\frac{n-1}{2}}|C_{0,2^{j}5\sigma}|\,|B_{\sqrt{\sigma}/16}(\w)|.
\]
We can now write
\begin{align*}
&\fint_{B_{\sqrt{\sigma}/16}(\w)}\fint_{B_{\sigma}(x)}\fint_{y-C_{j,\sigma}}|\wt{\theta}_{\nu,\sigma}(D)\wt{f}(z)|^{2}\ud z\ud y\ud\nu\\
&\lesssim \int_{B_{\sqrt{\sigma}/16}(\w)}\int_{\Rn}\ind_{x+C_{0,2^{j}5\sigma}}(z)|\wt{\theta}_{\nu,\sigma}(D)\wt{f}(z)|^{2}\frac{\ud z}{|C_{0,2^{j}5\sigma}|}\frac{\ud\nu}{|B_{\sqrt{\sigma}/16}(\w)|}\\
&\lesssim 2^{j\frac{n-1}{2}}\int_{S^{*}(\Rn)}\ind_{(x+C_{0,2^{j}5\sigma})\times B_{\sqrt{\sigma}/16}(\w)}(z,\nu)|\wt{\theta}_{\nu,\sigma}(D)\wt{f}(z)|^{2}\frac{\ud z\ud\nu}{|B_{c2^{j/2}\sqrt{\sigma}}(x,\w)|}\\
&\leq 2^{j\frac{n-1}{2}}\fint_{B_{c2^{j/2}\sqrt{\sigma}}(x,\w)}|\wt{\theta}_{\nu,\sigma}(D)\wt{f}(z)|^{2}\ud z\ud\nu.
\end{align*}
By plugging this into \eqref{eq:bigone}, we obtain
\begin{align*}
&\int_{0}^{1}\fint_{B_{\sqrt{\sigma}/16}(\w)}\fint_{B_{\sigma}(x)}\Big(\int_{C_{j,\sigma}}|\F^{-1}(\eta_{\w,\nu,\sigma})(z)\wt{\theta}_{\nu,\sigma}(D)\wt{f}(y-z)|\ud z\Big)^{2}\ud y\ud\nu\frac{\ud\sigma}{\sigma}\\
&\lesssim 2^{-2jN}\int_{0}^{1}\fint_{B_{\sqrt{\sigma}/16}(\w)}\fint_{B_{\sigma}(x)}\fint_{y-C_{j,\sigma}}|\wt{\theta}_{\nu,\sigma}(D)\wt{f}(z)|^{2}\ud z\ud y\ud\nu\frac{\ud\sigma}{\sigma}\\
&\lesssim 2^{-2j(N-\frac{n-1}{4})}\int_{0}^{1}\fint_{B_{c2^{j/2}\sqrt{\sigma}}(x,\w)}|\wt{\theta}_{\nu,\sigma}(D)\wt{f}(z)|^{2}\ud z\ud\nu\frac{\ud\sigma}{\sigma},
\end{align*}
where the implicit constant is independent of $x$, $\w$ and $j$.

Next, we combine what we have shown:
\begin{align*}
&\Big(\int_{0}^{\infty}\fint_{B_{\sigma}(x)}|\Psi_{\sigma}(D)\ph_{\w}(D)\wt{f}(y)|^{2}\ud y\frac{\ud\sigma}{\sigma}\Big)^{\frac{1}{2}}\\
&\lesssim \sum_{j=0}^{\infty}\!\Big(\int_{0}^{1}\fint_{B_{\sqrt{\sigma}/16}(\w)}\fint_{B_{\sigma}(x)}\!\Big(\!\int_{C_{j,\sigma}}\!|\F^{-1}(\eta_{\w,\nu,\sigma})(z)\wt{\theta}_{\nu,\sigma}(D)\wt{f}(y-z)|\ud z\!\Big)^{2}\ud y\ud\nu\frac{\ud\sigma}{\sigma}\Big)^{\frac{1}{2}}\\
&\lesssim \sum_{j=0}^{\infty}2^{-j(N-\frac{n-1}{4})}\Big(\int_{0}^{1}\fint_{B_{c2^{j/2}\sqrt{\sigma}}(x,\w)}|\wt{\theta}_{\nu,\sigma}(D)\wt{f}(z)|^{2}\ud z\ud\nu\frac{\ud\sigma}{\sigma}\Big)^{\frac{1}{2}}.
\end{align*}
Finally, to conclude the proof of \eqref{eq:1-qterm}, and thereby also of \eqref{eq:righthand}, we take $L^{p}(\Sp)$ norms. Then we use the change of aperture formula in Lemma \ref{lem:aperture}, we recall that $N>\frac{n-1}{4}+\frac{n}{2}$, and we apply \eqref{eq:newnorm}:
\begin{align*}
&\Big(\int_{\Sp}\Big(\int_{0}^{\infty}\fint_{B_{\sigma}(x)}|\Psi_{\sigma}(D)\ph_{\w}(D)\wt{f}(y)|^{2}\ud y\frac{\ud\sigma}{\sigma}\Big)^{\frac{p}{2}}\ud x\ud\w\Big)^{\frac{1}{p}}\\
&\lesssim \sum_{j=0}^{\infty}2^{-j(N-\frac{n-1}{4})}\Big(\int_{\Sp}\Big(\int_{0}^{1}\fint_{B_{c2^{j/2}\sqrt{\sigma}}(x,\w)}|\wt{\theta}_{\nu,\sigma}(D)\wt{f}(z)|^{2}\ud z\ud\nu\frac{\ud\sigma}{\sigma}\Big)^{\frac{p}{2}}\ud x\ud\w\Big)^{\frac{1}{p}}\\
&\lesssim \sum_{j=0}^{\infty}2^{-j(N-\frac{n-1}{4}-\frac{n}{2})}\Big(\int_{\Sp}\Big(\int_{0}^{1}\fint_{B_{\sqrt{\sigma}}(x,\w)}|\wt{\theta}_{\nu,\sigma}(D)\wt{f}(z)|^{2}\ud z\ud\nu\frac{\ud\sigma}{\sigma}\Big)^{\frac{p}{2}}\ud x\ud\w\Big)^{\frac{1}{p}}\\
&\lesssim \Big(\int_{\Sp}\Big(\int_{0}^{1}\fint_{B_{\sqrt{\sigma}}(x,\w)}|\wt{\theta}_{\nu,\sigma}(D)\wt{f}(z)|^{2}\ud z\ud\nu\frac{\ud\sigma}{\sigma}\Big)^{\frac{p}{2}}\ud x\ud\w\Big)^{\frac{1}{p}}\\
&\lesssim  \|\wt{f}\|_{\Hp}=\|(1-q)(D)f\|_{\Hp}\lesssim \|f\|_{\Hp},
\end{align*}
where in the final step we also used that $(1-q)(D)\in \La(\Hp)$, by \cite[Theorem 6.10]{HaPoRo18}. We have now proved half of the required statements.

For the remaining statements, we have to show that if $f\in\Sw'(\Rn)$ satisfies $q(D)f\in L^{p}(\Rn)$, $\ph_{\w}(D)f\in L^{p}(\Rn)$ for almost all $\w\in S^{n-1}$, and 
\[
\Big(\int_{\Sp}|\ph_{\w}(D)f(x)|^{p}\ud x\ud\w\Big)^{\frac{1}{p}}<\infty,
\]
then $f\in\Hp$ with 
\[
\|f\|_{\Hp}\lesssim\|q(D)f\|_{L^{p}(\Rn)}+\Big(\int_{\Sp}|\ph_{\w}(D)f(x)|^{p}\ud x\ud\w\Big)^{\frac{1}{p}}. 
\]
To this end, first note that the Sobolev embeddings for $\Hp$ from \cite[Theorem 7.4]{HaPoRo18} yield $q(D)f\in \Hp$, with
\[
\|q(D)f\|_{\Hp}\lesssim \|\lb D\rb^{\frac{n-1}{2}|\frac{1}{p}-\frac{1}{2}|}q(D)f\|_{L^{p}(\Rn)}\lesssim \|q(D)f\|_{L^{p}(\Rn)},
\]
where we also used that $q\in C^{\infty}_{c}(\Rn)$. 

To show that also $(1-q)(D)f\in\Hp$, we will use \eqref{eq:righthand} and duality. In particular, we will show that
\begin{equation}\label{eq:dualityRn}
|\lb (1-q)(D)f,g\rb_{\Rn}|\lesssim \Big(\int_{S^{n-1}}\|\ph_{\w}(D)(1-q)(D)f\|^{p}_{L^{p}(\Rn)}\ud\w\Big)^{\frac{1}{p}}\|g\|_{\HT^{p'}_{FIO}(\Rn)}
\end{equation}
for all $g\in\Sw(\Rn)$, with an implicit constant that is independent of $f$ and $g$. Here $\lb (1-q)(D)f,g\rb_{\Rn}$ denotes the duality between $(1-q)(D)f\in\Sw'(\Rn)$ and $g\in\Sw(\Rn)$. Since $\Sw(\Rn)\subseteq\HT^{p'}_{FIO}(\Rn)$ is dense (see \cite[Proposition 6.6]{HaPoRo18}), this then shows that $(1-q)(D)f\in (\Hpp)^{*}$ with
\begin{align*}
\|(1-q)(D)f\|_{(\Hpp)^{*}}&\lesssim \Big(\int_{S^{n-1}}\|\ph_{\w}(D)(1-q)(D)f\|^{p}_{L^{p}(\Rn)}\ud\w\Big)^{\frac{1}{p}}\\
&\lesssim\Big(\int_{\Sp}|\ph_{\w}(D)f(x)|^{p}\ud x\ud\w\Big)^{\frac{1}{p}},
\end{align*}
where in the final step we also used that $(1-q)(D)\in \La(L^{p}(\Rn))$, since $q\in C^{\infty}_{c}(\Rn)$. This will suffice to complete the proof, because $\Hp=(\Hpp)^{*}$ with equivalent norms, by \cite[Proposition 6.8]{HaPoRo18}.

Let $\tilde{q}\in C^{\infty}_{c}(\Rn)$ be real-valued and such that $\tilde{q}(\zeta)=0$ for $|\zeta|\geq 2$ and $\tilde{q}(\zeta)=1$ for $|\zeta|\leq \frac{1}{2}$. Set
\[
m(\zeta):=\begin{cases}
(1-\tilde{q}(\zeta))\big(\int_{S^{n-1}}\phi_{\nu}(\zeta)^{2}\ud\nu\big)^{-1}&\text{for }\zeta\in\supp(1-\wt{q}),\\
0&\text{otherwise}.
\end{cases}
\]
It then follows from \eqref{eq:boundbelow} that $m\in C^{\infty}(\Rn)$ with
\begin{equation}\label{eq:boundsm}
|\partial_{\zeta}^{\alpha}m(\zeta)|\lesssim (1+|\zeta|)^{-|\alpha|}
\end{equation}
for all $\alpha\in\Z_{+}^{n}$ and $\zeta\in\Rn$. Also note that $(1-q)(D)=(1-q)(D)(1-\tilde{q})(D)$. 

Now let $g\in\Sw(\Rn)$. To prove \eqref{eq:dualityRn}, we will in fact work with the duality between $V(1-q)(D)f$ and $Ug$ from \eqref{eq:duality}, where we note that $Ug\in \Da(\Spp)$ by Lemma \ref{lem:transforms}. More precisely, since $\Psi_{\sigma}(D)(1-q)(D)f=0$ for $\sigma\geq 1$, \eqref{eq:repro} yields
\begin{align*}
&\lb (1-q)(D)f,g\rb_{\Rn}=\lb V(1-q)(D)f,Ug\rb_{\Spp}\\
&=\int_{0}^{1}\int_{S^{n-1}}\int_{\Rn}\theta_{\w,\sigma}(D)(1-q)(D)f(x)\overline{\chi_{\w,\sigma}(D)g(x)}\ud x\ud\w\frac{\ud\sigma}{\sigma}.
\end{align*}
Next, note that
\begin{align*}
&\int_{0}^{1}\int_{S^{n-1}}\int_{\Rn}\theta_{\w,\sigma}(D)(1-q)(D)f(x)\overline{\chi_{\w,\sigma}(D)g(x)}\ud x\ud\w\frac{\ud\sigma}{\sigma}\\
&=\int_{0}^{1}\int_{S^{n-1}}\int_{\Rn}\theta_{\w,\sigma}(D)(1-q)(D)f(x)\overline{(1-\tilde{q})(D)\chi_{\w,\sigma}(D)g(x)}\ud x\ud\w\frac{\ud\sigma}{\sigma}\\
&=\int_{0}^{1}\int_{S^{n-1}}\int_{\Rn}\theta_{\w,\sigma}(D)(1-q)(D)f(x)\overline{\theta_{\w,\sigma}(D)m(D)g(x)}\ud x\ud\w\frac{\ud\sigma}{\sigma}\\
&=\int_{S^{n-1}}\int_{0}^{1}\int_{\Rn}\Psi_{\sigma}(D)\ph_{\w}(D)(1-q)(D)f(x)\overline{\Psi_{\sigma}(D)\ph_{\w}(D)m(D)g(x)}\ud x\frac{\ud\sigma}{\sigma}\ud\w\\
&=\int_{S^{n-1}}\lb\ph_{\w}(D)(1-q)(D)f,\ph_{\w}(D)m(D)g\rb_{\Rn}\ud\w,
\end{align*} 
where for the final step we used \eqref{eq:Psi}.
Moreover, it follows from \eqref{eq:boundsm} and \cite[Theorem 6.10]{HaPoRo18} that $m(D)g\in\HT^{p'}_{FIO}(\Rn)$ with $\|m(D)g\|_{\HT^{p'}_{FIO}(\Rn)}\lesssim \|g\|_{\HT^{p'}_{FIO}(\Rn)}$. Hence, applying \eqref{eq:righthand} to $m(D)g$, with $p$ replaced by $p'$, we see that
\begin{align*}
&|\lb (1-q)(D)f,g\rb_{\Rn}|=\Big|\int_{0}^{1}\int_{S^{n-1}}\int_{\Rn}\theta_{\w,\sigma}(D)(1-q)(D)f(x)\overline{\chi_{\w,\sigma}(D)g(x)}\ud x\ud\w\frac{\ud\sigma}{\sigma}\Big|\\
&\leq \int_{S^{n-1}}|\lb\ph_{\w}(D)(1-q)(D)f,\ph_{\w}(D)m(D)g\rb_{\Rn}|\ud\w\\
&\leq \int_{S^{n-1}}\|\ph_{\w}(D)(1-q)(D)f\|_{L^{p}(\Rn)}\|\ph_{\w}(D)m(D)g\|_{L^{p'}(\Rn)}\ud\w\\
&\leq \Big(\int_{S^{n-1}}\|\ph_{\w}(D)(1-q)(D)f\|^{p}_{L^{p}(\Rn)}\ud\w\Big)^{\frac{1}{p}}\Big(\int_{S^{n-1}}\|\ph_{\w}(D)m(D)g\|_{L^{p'}(\Rn)}^{p'}\ud\w\Big)^{\frac{1}{p'}}\\
&\lesssim \Big(\int_{S^{n-1}}\|\ph_{\w}(D)(1-q)(D)f\|^{p}_{L^{p}(\Rn)}\ud\w\Big)^{\frac{1}{p}}\|m(D)g\|_{\HT^{p'}_{FIO}(\Rn)}\\
&\lesssim \Big(\int_{S^{n-1}}\|\ph_{\w}(D)(1-q)(D)f\|^{p}_{L^{p}(\Rn)}\ud\w\Big)^{\frac{1}{p}}\|g\|_{\HT^{p'}_{FIO}(\Rn)}.
\end{align*}
This concludes the proof of \eqref{eq:dualityRn} and of the theorem.
\end{proof}

\begin{remark}\label{rem:alternateproof}
For $p\in(1,2]$, an alternative proof of \eqref{eq:righthand} can be obtained by combining the vertical square function characterization of $L^{p}(\Rn)$ (see \eqref{eq:vertical} below) with bounds for vertical square functions in terms of conical ones (see \cite[Proposition 2.1 and Remark 2.2]{AuHoMa12} or \cite[Equation (2.9)]{HaPoRo18}):
\begin{align*}
&\|q(D)f\|_{L^{p}(\Rn)}+\Big(\int_{S^{n-1}}\|\ph_{\w}(D)f\|_{L^{p}(\Rn)}^{p}\ud\w\Big)^{1/p}\\
&\lesssim \|q(D)f\|_{L^{p}(\Rn)}+\Big(\int_{S^{n-1}}\int_{\Rn}\Big(\int_{0}^{1}|\Psi_{\sigma}(D)\ph_{\w}(D)f(x)|^{2}\frac{\ud\sigma}{\sigma}\Big)^{p/2}\ud x\ud\w\Big)^{1/p}\\
&\lesssim \|q(D)f\|_{L^{p}(\Rn)}+\Big(\int_{\Sp}\Big(\int_{0}^{1}\fint_{B_{\sqrt{\sigma}}(x,\w)}|\theta_{\nu,\sigma}(D)f(y)|^{2}\ud y\ud\nu\frac{\ud\sigma}{\sigma}\Big)^{p/2}\ud x\ud\w\Big)^{1/p}\\
&\lesssim \|f\|_{\Hp}.
\end{align*}
Here we also used \eqref{eq:lowfreq} with $m=\ph_{\w}q$, and in the final step we used that $q\in C^{\infty}_{c}(\Rn)$ and we applied the Sobolev embeddings from \cite[Theorem 7.4]{HaPoRo18}. The same argument in fact works for $p=1$ with the classical Hardy space $H^{1}(\Rn)$, and it shows that
\begin{equation}\label{eq:1version}
\|q(D)f\|_{L^{1}(\Rn)}+\int_{\Sp}\|\ph_{\w}(D)f\|_{H^{1}(\Rn)}\ud\w\lesssim \|f\|_{\HT^{1}_{FIO}(\Rn)}.
\end{equation}
However, this argument does not work for $p>2$, nor can one obtain the reverse inequality in \eqref{eq:righthand} in this manner for $p<2$.
\end{remark}

\begin{remark}\label{rem:proof}
One of the main technical difficulties in the proof of Theorem \ref{thm:chargeneral} is to move from averages over isotropic balls $B_{\sigma}(x)\subseteq \Rn$, which arise from the conical square function characterization of $L^{p}(\Rn)$ in \eqref{eq:1-qterm}, to averages over the anisotropic balls $B_{\sqrt{\sigma}}(x,\w)\subseteq\Sp$ that occur in the $\Hp$-norm. Our proof of \eqref{eq:righthand} works because the projection of a ball $B_{\sqrt{\sigma}}(x,\w)$ onto $\Rn$ is no smaller, up to constants independent of $x$, $\w$, and $\sigma$, than the ball $B_{\sigma}(x)$; this fact is crucial in \eqref{eq:removal}. 

When attempting to prove the reverse inequality in \eqref{eq:1-qterm} directly, one has to replace the anisotropic balls which arise by projecting $B_{\sqrt{\sigma}}(x,\w)$ onto $\Rn$ by the fundamentally smaller balls $B_{\sigma}(x)$. This appears to be problematic, hence for $p\in(1,\infty)$ we instead used duality. One could attempt to prove the reverse inequality in \eqref{eq:1version} in a similar way, by proving its dual inequality
\begin{equation}\label{eq:inftyversion}
\|q(D)f\|_{L^{\infty}(\Rn)}+\esssup_{(x,\w)\in\Sp}\|\ph_{\w}(D)f\|_{\BMO(\Rn)}\lesssim \|f\|_{\HT^{\infty}_{FIO}(\Rn)}.
\end{equation}
However, the conical square function characterization of $\BMO(\Rn)$ involves a norm, the $T^{\infty}(\Rn)$ norm, that is similar to that in \eqref{eq:inftynorm} and of a different nature than the conical square function norm that one encounters for $p<\infty$. The arguments that we used for $p<\infty$ do not appear to extend directly to the case where $p=\infty$. We do not know whether \eqref{eq:inftyversion} or the reverse inequality in \eqref{eq:1version} hold.
\end{remark}

\subsection{Additional results}

By combining Theorem \ref{thm:chargeneral} with known characterizations of $L^{p}(\Rn)$, one obtains various characterizations of $\Hp$ for $1<p<\infty$. We give just two here: a maximal function characterization and one in terms of vertical square functions. 

For the maximal function characterization, let $\Phi\in \Sw(\Rn)$ be such that $\Phi(0)=1$, and for $\sigma>0$ and $\zeta\in\Rn$ set $\Phi_{\sigma}(\zeta)=\Phi(\sigma\zeta)$ as before. Recall that $f\in\Sw'(\Rn)$ is a \emph{bounded distribution} if $f\ast g\in L^{\infty}(\Rn)$ for all $g\in\Sw(\Rn)$.

\begin{corollary}\label{cor:maximal}
Let $p\in(1,\infty)$. Then there exists a constant $C>0$ such that the following holds for all $f\in\Sw'(\Rn)$. One has $f\in\Hp$ if and only if $q(D)f\in L^{p}(\Rn)$, $\ph_{\w}(D)f$ is a bounded distribution for almost all $\w\in S^{n-1}$, and 
\[
\Big(\int_{\Sp}\sup_{\sigma>0}|\Phi_{\sigma}(D)\ph_{\w}(D)f(x)|^{p}\ud x\ud\w\Big)^{\frac{1}{p}}<\infty.
\]
Moreover, if $f\in\Hp$ then
\begin{align*}
\frac{1}{C}\|f\|_{\Hp}&\leq \|q(D)f\|_{L^{p}(\Rn)}+\Big(\int_{\Sp}\sup_{\sigma>0}|\Phi_{\sigma}(D)\ph_{\w}(D)f(x)|^{p}\ud x\ud\w\Big)^{\frac{1}{p}}\\
&\leq C\|f\|_{\Hp}.
\end{align*}
\end{corollary}
\begin{proof}
Since
\[
\Big(\int_{\Sp}|\ph_{\w}(D)f(x)|^{p}\ud x\ud\w\Big)^{1/p}=\Big(\int_{S^{n-1}}\|\ph_{\w}(D)f\|_{L^{p}(\Rn)}^{p}\ud\w\Big)^{1/p},
\]
the corollary follows from Theorem \ref{thm:chargeneral} and the maximal function characterization of $L^{p}(\Rn)$ for $1<p<\infty$ (see \cite[Theorems 2.1.2 and 2.1.4]{Grafakos14b}):
\[
\|\ph_{\w}(D)f\|_{L^{p}(\Rn)}\eqsim \Big(\int_{\Rn}\sup_{\sigma>0}|\Phi_{\sigma}(D)\ph_{\w}(D)f(x)|^{p}\ud x\Big)^{1/p}
\]
for $w\in S^{n-1}$.
\end{proof}

Next, we characterize $\Hp$ in terms of vertical square functions. 

\begin{corollary}\label{cor:vertical}
Let $p\in(1,\infty)$. Then there exists a constant $C>0$ such that the following holds for all $f\in\Sw'(\Rn)$. One has $f\in\Hp$ if and only if $q(D)f\in L^{p}(\Rn)$ and 
\[
\Big(\int_{\Sp}\Big(\int_{0}^{1}|\theta_{\w,\sigma}(D)f(x)|^{2}\frac{\ud\sigma}{\sigma}\Big)^{\frac{p}{2}}\ud x\ud\w\Big)^{\frac{1}{p}}<\infty.
\]
Moreover, if $f\in \Hp$ then
\begin{align*}
\frac{1}{C}\|f\|_{\Hp}&\leq \|q(D)f\|_{L^{p}(\Rn)}+\Big(\int_{\Sp}\Big(\int_{0}^{1}|\theta_{\w,\sigma}(D)f(x)|^{2}\frac{\ud\sigma}{\sigma}\Big)^{\frac{p}{2}}\ud x\ud\w\Big)^{\frac{1}{p}}\\
&\leq C\|f\|_{\Hp}.
\end{align*}
\end{corollary}
\begin{proof}
By \cite[Sections 1.5.1 and 2.4.2]{Triebel92}, a $g\in\Sw'(\Rn)$ satisfies $g\in L^{p}(\Rn)$ if and only if $q(D)g\in L^{p}(\Rn)$ and $\int_{\Rn}(\int_{0}^{1}|\Psi_{\sigma}(D)g(x)|^{2}\frac{\ud\sigma}{\sigma})^{p/2}\ud x<\infty$, in which case
\begin{equation}\label{eq:vertical}
\|g\|_{L^{p}(\Rn)}\eqsim \|q(D)g\|_{L^{p}(\Rn)}+\Big(\int_{\Rn}\Big(\int_{0}^{1}|\Psi_{\sigma}(D)g(x)|^{2}\frac{\ud\sigma}{\sigma}\Big)^{\frac{p}{2}}\ud x\Big)^{\frac{1}{p}}.
\end{equation}
With $g=\ph_{\w}(D)f$ for $\w\in S^{n-1}$, one can now use Theorem \ref{thm:chargeneral} to obtain the desired conclusion:
\begin{align*}
&\|f\|_{\Hp}\eqsim\|q(D)f\|_{L^{p}(\Rn)}+\Big(\int_{S^{n-1}}\|\ph_{\w}(D)f\|_{L^{p}(\Rn)}^{p}\ud\w\Big)^{\frac{1}{p}}\\
&\eqsim \|q(D)f\|_{L^{p}(\Rn)}+\Big(\int_{\Sp}\Big(\int_{0}^{1}|\Psi_{\sigma}(D)\ph_{\w}(D)f(x)|^{2}\frac{\ud\sigma}{\sigma}\Big)^{\frac{p}{2}}\ud x\ud\w\Big)^{\frac{1}{p}},
\end{align*}
where we also used that $\|q(D)\ph_{\w}(D)f\|_{L^{p}(\Rn)}\lesssim \|q(D)f\|_{L^{p}(\Rn)}$ for an implicit constant independent of $w\in S^{n-1}$.
\end{proof}

\section*{Acknowledgements} 

The author would like to thank Andrew Hassell and Pierre Portal for numerous helpful suggestions, and the anonymous referee for useful comments that have helped improved the article.

\bibliographystyle{plain}
\bibliography{Bibliography}

\end{document}